\documentclass[11pt]{amsart}

\usepackage{a4wide}
\usepackage{cite}
\usepackage{stmaryrd}
\usepackage{hyperref}
\usepackage[T1]{fontenc}

\usepackage{amsthm, amsmath, amsfonts}
\usepackage{mathtools}

\usepackage{enumerate}
\usepackage{enumitem}
\usepackage[left]{lineno}

\usepackage{graphicx}
\usepackage{pgfplots}
\pgfplotsset{compat=1.12}
\usepackage{pgfplotstable}
\usepackage{caption}
\usepackage{subcaption}

\usepackage{fontawesome5}

\allowdisplaybreaks
\usepackage{tikz}
\usetikzlibrary{arrows,backgrounds,calc}
\usetikzlibrary{math,arrows,backgrounds}
\usetikzlibrary{decorations.pathreplacing}
\usetikzlibrary{fit}

\newtheorem{lemma}{Lemma}[section]
\newtheorem{proposition}[lemma]{Proposition}
\newtheorem{theorem}[lemma]{Theorem}
\newtheorem{maintheorem}[lemma]{Theorem}
\newtheorem{corollary}[lemma]{Corollary}
\theoremstyle{remark}

\newtheorem*{definition*}{Definition}
\newtheorem{remark}[lemma]{Remark}

\DeclarePairedDelimiter{\abs}{|}{|}
\DeclareMathOperator{\Arg}{Arg}
\newcommand{\C}{\mathbb{C}}
\newcommand{\N}{\mathbb{N}}
\newcommand{\R}{\mathbb{R}}

\title[]{The distribution of the maximum protection number in simply generated trees}

\author[C. Heuberger \and S. J. Selkirk \and S. Wagner]{Clemens Heuberger \and Sarah J. Selkirk \and Stephan Wagner}

\address[Clemens Heuberger, Sarah J. Selkirk] {Institut
  f\"ur Mathematik, Alpen-Adria-Uni\-ver\-si\-t\"at Klagenfurt,
  Universit\"atsstra\ss e 65--67, 9020 Klagenfurt, Austria. }
\email{\href{mailto:clemens.heuberger@aau.at}{clemens.heuberger (at) aau.at}, \href{mailto:sarah.selkirk@aau.at}{sarah.selkirk (at) aau.at}}
\address[Stephan Wagner] {Department of Mathematics, Uppsala Universitet, Box 480, 751 06 Uppsala, Sweden, and Institute of Discrete Mathematics, TU Graz, Steyrergasse 30, 8010 Graz, Austria.}
\email{\href{mailto:stephan.wagner@math.uu.se}{stephan.wagner (at) math.uu.se}}
\thanks{The research of C.~Heuberger and S.~J.~Selkirk was funded in part by the
  Austrian Science Fund (FWF) [10.55776/P28466], \emph{Analytic Combinatorics: Digits, Automata and Trees} and
  Austrian Science Fund (FWF) [10.55776/DOC78]. S.~Wagner is supported by the Knut and Alice Wallenberg Foundation, grant KAW 2017.0112, and the Swedish research council (VR), grant  2022-04030. For open access purposes, the authors have applied a CC BY public copyright license to any author-accepted manuscript version arising from this submission.}


\begin{document}
\pagestyle{plain}

\begin{abstract}
  The protection number of a vertex $v$ in a tree is the length of the shortest
  path from $v$ to any leaf contained in the maximal subtree where $v$ is the
  root. In this paper, we determine the distribution of the maximum protection
  number of a vertex in simply generated trees, thereby refining a recent
  result of Devroye, Goh and Zhao. Two different cases can be observed: if the
  given family of trees allows vertices of outdegree $1$, then the maximum
  protection number is on average logarithmic in the tree size, with a discrete
  double-exponential limiting distribution. If no such vertices are allowed,
  the maximum protection number is doubly logarithmic in the tree size and
  concentrated on at most two values. These results are obtained by studying
  the singular behaviour of the generating functions of trees with bounded
  protection number.
  While a general distributional result by Prodinger and Wagner can be used in the first case,
  we prove a variant of that result in the second case.
\end{abstract}

\maketitle

\section{Introduction}

\subsection{Simply generated trees}

Simply generated trees were introduced by Meir and Moon \cite{Meir-Moon:1978:simply-gen}, and owing to 
their use in describing an entire class of trees, have created a general framework for studying random 
trees. A simply generated family of rooted trees is characterised by a sequence of weights associated 
with the different possible outdegrees of a vertex. Specifically, for a given sequence of nonnegative 
real numbers $w_j$ ($j \geq 0$), one defines the weight of a rooted ordered tree to be the product 
$\prod_v w_{d(v)}$ over all vertices of the tree, where $d(v)$ denotes the outdegree (number of 
children) of $v$. Letting $\Phi(t) = \sum_{j \geq 0} w_j t^j$ be the weight generating function and 
$Y(x)$ the generating function in which the coefficient of $x^n$ is the sum of the weights over all 
$n$-vertex rooted ordered trees, one has the fundamental relation 
\begin{equation}\label{eq:Y_functional_equation}
    Y(x) = x\Phi(Y(x)).
\end{equation}
Common examples of simply generated trees are: plane trees with weight generating function 
$\Phi(t) = 1/(1-t)$; binary trees ($\Phi(x) = 1+t^2$); pruned binary trees ($\Phi(t) = (1+t)^2$); and 
labelled trees ($\Phi(t) = e^t$). In the first three examples, $Y(x)$ becomes an ordinary generating 
function with the total weight being the number of trees in the respective family, while $Y(x)$ can 
be seen as an exponential generating function in the case of labelled trees.

In addition to the fact that the notion of simply generated trees covers many important examples, 
there is also a strong connection to the probabilistic model of Bienaym\'e--Galton--Watson trees: here, 
one fixes a probability distribution on the set of nonnegative integers. Next, a random tree is 
constructed by starting with a root that produces offspring according to the given distribution. In 
each subsequent step, all vertices of the current generation also produce offspring according to the 
same distribution, all independent of each other and independent of all previous generations.  The 
process stops if none of the vertices of a generation have children. If the weights in the construction 
of a simply generated family are taken to be the corresponding probabilities of the offspring 
distribution, then one verifies easily that the distribution of a random $n$-vertex tree from that 
family (with probabilities proportional to the weights) is the same as that of the 
Bienaym\'e--Galton--Watson process, conditioned on the event that the final tree has $n$ vertices. 

Conversely, even if the weight sequence of a simply generated family does not represent a probability 
measure, it is often possible to determine an equivalent probability measure that produces the same 
random tree distribution. For example, random plane trees correspond to a geometric distribution 
while random rooted labelled trees correspond to a Poisson distribution. We refer 
to~\cite{Drmota:2009:random} and \cite{Janson:2012:simply-generated-survey} for more background on 
simply generated trees and Bienaym\'e--Galton--Watson trees.

\subsection{Protection numbers in trees}

Protection numbers in trees measure the distance to the nearest leaf successor. Formally, this can 
be expressed as follows.

\begin{definition*}[Protection number]
    The \emph{protection number of a vertex $v$} is the length of the shortest path from $v$ to any leaf 
    contained in the maximal subtree where $v$ is the root.
\end{definition*}
Alternatively, the protection number can be defined recursively: a leaf has protection number $0$, the 
parent of a leaf has protection number $1$, and generally the protection number of an interior vertex 
is the minimum of the protection numbers of its children plus $1$. In this paper, we will be 
particularly interested in the \emph{maximum protection number} of a tree, which is the largest 
protection number among all vertices. Figure~\ref{fig:tree} shows an example of a tree along with the 
protection numbers of all its vertices.

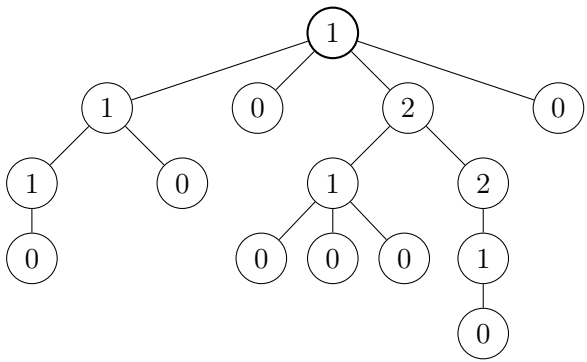
\begin{figure}[ht]
    \begin{tikzpicture}[scale = 0.5]
        \node[circle, draw, thick] (0) at (0, 0) {1};
        \node[circle, draw] (1) at (6, -2) {0};
        \node[circle, draw] (2) at (2, -2) {2};
        \node[circle, draw] (3) at (-2, -2) {0};
        \node[circle, draw] (4) at (-6, -2) {1};
        
        \node[circle, draw] (5) at (-4, -4) {0};
        \node[circle, draw] (6) at (-8, -4) {1};
        \node[circle, draw] (7) at (-8, -6) {0};
        
        \node[circle, draw] (8) at (0, -4) {1};
        \node[circle, draw] (9) at (4, -4) {2};

        \node[circle, draw] (10) at (4, -6) {1};
        \node[circle, draw] (11) at (4, -8) {0};
        
        \node[circle, draw] (12) at (-1.9, -6) {0};
        \node[circle, draw] (13) at (0, -6) {0};
        \node[circle, draw] (14) at (1.9, -6) {0};
        
        \draw[-] (0) to (1);
        \draw[-] (0) to (2);
        \draw[-] (0) to (3);
        \draw[-] (0) to (4);
        
        \draw[-] (2) to (8);
        \draw[-] (2) to (9);
        
        \draw[-] (8) to (12);
        \draw[-] (8) to (13);
        \draw[-] (8) to (14);

        \draw[-] (9) to (10);
        \draw[-] (10) to (11);

        \draw[-] (4) to (5);
        \draw[-] (4) to (6);
        
        \draw[-] (6) to (7);
    
    \end{tikzpicture}
    \caption{A plane tree with 15 vertices and the protection number of each vertex indicated. The maximum protection 
    number of this tree is $2$.}
    \label{fig:tree}
\end{figure}

The study of protection numbers in trees began with Cheon and Shapiro \cite{Cheon-Shapiro:2008:protec} considering the average number of vertices with 
protection number of at least 2 (called $2$-protected) in ordered trees. Several other authors contributed to knowledge in this direction, 
by studying the number of $2$-protected vertices in various types of trees: $k$-ary trees~\cite{Mansour:2011:protec}; digital search 
trees~\cite{Du-Prodinger:2012:notes}; binary search trees~\cite{Mahmoud-Ward:2012:protec}; ternary search trees~\cite{Holmgren-Janson:2015:asymp}; 
tries and suffix trees~\cite{Gaither-Homma-Sellke-Ward:2012:protec}; random recursive trees~\cite{Mahmoud-Ward:2015:asymp}; and general simply generated trees 
from which some previously known cases were also obtained~\cite{Devroye-Janson:2014:protec}.

Generalising the concept of a vertex being $2$-protected, $k$-protected
vertices---when a vertex has protection number at least $k$---also became a
recent topic of interest. Devroye and Janson \cite{Devroye-Janson:2014:protec}
proved convergence of the probability that a random vertex in a random simply
generated tree has protection number $k$. Copenhaver gave a closed formula for
the number of $k$-protected vertices in all unlabelled rooted plane trees on
$n$ vertices along with expected values~\cite{Copenhaver:2016}, and these results were extended by
Heuberger and
Prodinger~\cite{Heuberger-Prodinger:2017:protec-number-plane-trees}. A study of
$k$-protected vertices in binary search trees was done by
B\'{o}na~\cite{Bona:2014} and B\'{o}na and
Pittel~\cite{Bona-Pittel:2017}. Holmgren and Janson
\cite{Holmgren-Janson:2015:limit-laws} proved general limit theorems for fringe
subtrees and related tree functionals, applications of which include a normal
limit law for the number of $k$-protected vertices in binary search trees and
random recursive trees.

Moreover, the protection number of the root of families of trees has also been studied. In~\cite{Heuberger-Prodinger:2017:protec-number-plane-trees}, 
Heuberger and Prodinger derived the probability of a plane tree having a root that is $k$-protected, the probability distribution of the 
protection number of the root of recursive trees is determined by Go{\l}{\k{e}}biewski and Klimczak in~\cite{Golebiewski-Klimczak:2019:protec}.
The protection number of the root in simply generated trees, P\'{o}lya trees, and unlabelled non-plane binary trees was studied by 
Gittenberger, Go{\l}{\k{e}}biewski, Larcher, and Sulkowska in~\cite{Gittenberger-Golebiewski-Larcher-Sulkowska:2021:protec}, where they also 
obtained results relating to the protection number of a randomly chosen vertex. 

Very recently, Devroye, Goh and Zhao~\cite{Devroye-Goh-Zhao:2023:peel} studied
the maximum protection number in Bienaym\'e--Galton--Watson trees, referring to
it as the leaf-height.  Specifically, they showed the following: if $X_n$ is
the maximum protection number in a Bienaym\'e--Galton--Watson tree conditioned
on having $n$ vertices, then $\frac{X_n}{\log n}$ converges in probability to a
constant if there is a positive probability that a vertex has exactly one
child. If this is not the case, then $\frac{X_n}{\log \log n}$ converges in
probability to a constant.

Our aim in this paper is to refine the result of Devroye, Goh and Zhao by
providing the full limiting distribution of the maximum protection number. For
our analytic approach, the framework of simply generated trees is more natural
than the probabilistic setting of Bienaym\'e--Galton--Watson trees, though as
mentioned earlier the two are largely equivalent.

\subsection{Statement of results}

As was already observed by Devroye, Goh and Zhao in
\cite{Devroye-Goh-Zhao:2023:peel}, there are two fundamentally different cases
to be considered, depending on whether or not vertices of outdegree $1$ are
allowed (have nonzero weight) in the given family of simply generated trees. If
such vertices can occur, then we find that the maximum protection number of a
random tree with $n$ vertices is on average of order $\log n$, with a discrete
double-exponential distribution in the limit.  On the other hand, if there are
no vertices of outdegree $1$, then the maximum protection number is on average
of order $\log \log n$. There is an intuitive explanation for this
phenomenon. If outdegree $1$ is allowed, it becomes easy to create vertices
with high protection number: if the subtree rooted at a vertex is an
$(h+1)$-vertex path, then this vertex has protection number $h$. On the other
hand, if outdegree $1$ is forbidden, then the smallest possible subtree rooted
at a vertex of protection number $h$ is a complete binary tree with $2^{h+1}-1$
vertices. An illustration of the two cases is given in Figure~\ref{fig:Phi0}.

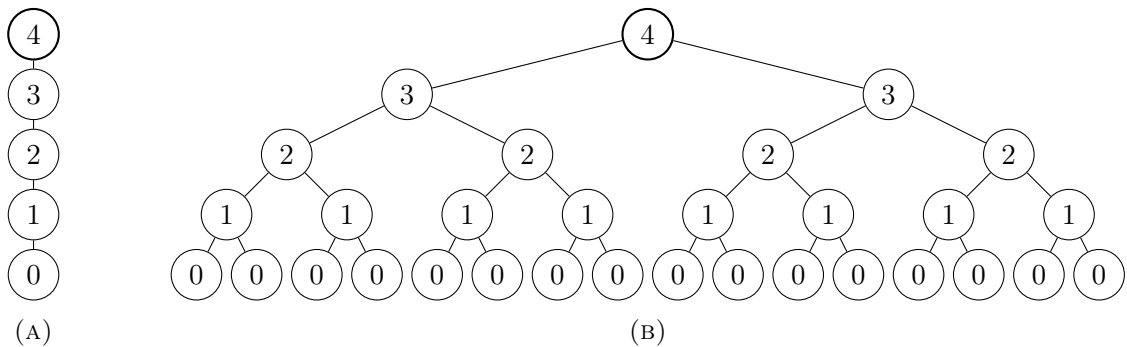
\begin{figure}[ht]

    \begin{subfigure}{0.1\textwidth}
      \centering
      \begin{tikzpicture}[scale = 0.4]
          \node[circle, draw, thick] (0) at (0, 0) {4};
          \node[circle, draw] (1) at (0, -2) {3};
          \node[circle, draw] (2) at (0, -4) {2};
          \node[circle, draw] (3) at (0, -6) {1};
          \node[circle, draw] (4) at (0, -8) {0};

          \draw[-] (0) to (1);
          \draw[-] (1) to (2);
          \draw[-] (2) to (3);
          \draw[-] (3) to (4);
        
      \end{tikzpicture}
      \caption{}
    \end{subfigure}
    \hfill
    \begin{subfigure}{0.85\textwidth}
      \centering
      \begin{tikzpicture}[scale = 0.4]
        \node[circle, draw, thick] (0) at (0, 0) {4};
        \node[circle, draw] (1) at (-8, -2) {3};
        \node[circle, draw] (2) at (8, -2) {3};
        \node[circle, draw] (3) at (-12, -4) {2};
        \node[circle, draw] (4) at (-4, -4) {2};
        \node[circle, draw] (5) at (4, -4) {2};
        \node[circle, draw] (6) at (12, -4) {2};
        \node[circle, draw] (7) at (-14, -6) {1};
        \node[circle, draw] (8) at (-10, -6) {1};
        \node[circle, draw] (9) at (-6, -6) {1};
        \node[circle, draw] (10) at (-2, -6) {1};
        \node[circle, draw] (11) at (2, -6) {1};
        \node[circle, draw] (12) at (6, -6) {1};
        \node[circle, draw] (13) at (10, -6) {1};
        \node[circle, draw] (14) at (14, -6) {1};
        
        \node[circle, draw] (15) at (-15, -8) {0};
        \node[circle, draw] (16) at (-13, -8) {0};
        \node[circle, draw] (17) at (-11, -8) {0};
        \node[circle, draw] (18) at (-9, -8) {0};
        \node[circle, draw] (19) at (-7, -8) {0};
        \node[circle, draw] (20) at (-5, -8) {0};
        \node[circle, draw] (21) at (-3, -8) {0};
        \node[circle, draw] (22) at (-1, -8) {0};
        \node[circle, draw] (23) at (1, -8) {0};
        \node[circle, draw] (24) at (3, -8) {0};
        \node[circle, draw] (25) at (5, -8) {0};
        \node[circle, draw] (26) at (7, -8) {0};
        \node[circle, draw] (27) at (9, -8) {0};
        \node[circle, draw] (28) at (11, -8) {0};
        \node[circle, draw] (29) at (13, -8) {0};
        \node[circle, draw] (30) at (15, -8) {0};

        \draw[-] (0) to (1);
        \draw[-] (0) to (2);
        \draw[-] (1) to (3);
        \draw[-] (1) to (4);
        \draw[-] (3) to (7);
        \draw[-] (3) to (8);
        \draw[-] (4) to (9);
        \draw[-] (4) to (10);
        \draw[-] (2) to (5);
        \draw[-] (2) to (6);
        \draw[-] (5) to (11);
        \draw[-] (5) to (12);
        \draw[-] (6) to (13);
        \draw[-] (6) to (14);

        \draw[-] (7) to (15);
        \draw[-] (7) to (16);
        \draw[-] (8) to (17);
        \draw[-] (8) to (18);
        \draw[-] (9) to (19);
        \draw[-] (9) to (20);
        \draw[-] (10) to (21);
        \draw[-] (10) to (22);
        \draw[-] (11) to (23);
        \draw[-] (11) to (24);
        \draw[-] (12) to (25);
        \draw[-] (12) to (26);
        \draw[-] (13) to (27);
        \draw[-] (13) to (28);
        \draw[-] (14) to (29);
        \draw[-] (14) to (30);

      \end{tikzpicture}
    \caption{}
  \end{subfigure}

    \caption{Smallest examples where a tree may (A) or may not (B) have exactly one child and the root has protection number $4$.}
    \label{fig:Phi0}
\end{figure}

In the case where vertices of outdegree $1$ can occur, the limiting distribution turns out to be a discrete double-exponential distribution that also occurs in many other combinatorial examples, and for which general results are available---see Section~\ref{sec:general_distributional}. These results are adapted in Section~\ref{sec:prowags-doubly-exp-restated} to the case where there are no vertices of outdegree $1$.

In the following results, we make a common technical assumption, stating formally that there is a positive real number $\tau$, less than the radius of convergence of $\Phi$, such that $\Phi(\tau) = \tau\Phi'(\tau)$ (see Section~\ref{sec:sg_trees_basic} for further details). This is equivalent to the offspring distribution of the associated Bienaym\'e--Galton--Watson process having a finite exponential moment, which is the case for all the examples mentioned earlier (plane trees, binary trees, pruned binary trees, labelled trees). This assumption is crucial for the analytic techniques that we are using, which are based on an asymptotic analysis of generating functions. However, it is quite likely that our main results remain valid under somewhat milder conditions.

\begin{maintheorem}
\label{thm:main_outdeg1_allowed}
Given a family of simply generated trees with $w_1 = \Phi'(0) \neq 0$, the proportion of trees of size $n$ whose maximum protection number is at most $h$ is asymptotically given by 
  \begin{equation*}
    \exp\big({-}\kappa n d^{-h}\big)(1 + o(1))
  \end{equation*}
  as $n \to \infty$ and $h = \log_d(n) + O(1)$, where $\kappa$ (given in~\eqref{eq:expontential-kappa}) and $d = (\rho\Phi'(0))^{-1} > 1$ are positive constants, with $\rho$ as defined in~\eqref{eq:fundamental-connection-between-greek-variables}. 
Moreover, the expected value of the maximum protection number in trees with $n$ vertices is
\begin{equation*}
\log_d(n) + \log_d(\kappa) + \frac{\gamma}{\log(d)} + \frac{1}{2} + \psi_d(\log_d(\kappa n)) + o(1),
\end{equation*}
  where $\gamma$ denotes the Euler--Mascheroni constant and $\psi_d$ is the $1$-periodic function that 
  is defined by the Fourier series 
  \begin{equation}\label{eq:psib}
    \psi_d(x) = -\frac{1}{\log(d)}\sum_{k \neq 0}\Gamma\Big({-}\frac{2k\pi i}{\log(d)}\Big)e^{2k\pi i x}.
  \end{equation}
\end{maintheorem}

In the case where vertices of outdegree $1$ are excluded, we show that the
maximum protection number is strongly concentrated. In fact, with high
probability it only takes on one of at most two different values (depending on
the size of the tree). The precise result can be stated as follows.

\begin{maintheorem}
\label{thm:main_outdeg1_forbidden}
Given a family of simply generated trees with $w_1 = \Phi'(0) = 0$, set
$r = \min \{i \in \mathbb{N}\colon i \geq 2 \text{ and } w_i \neq 0\}$ and
$D=\gcd\{i\in \N \colon w_i\neq 0\}$. The
proportion of trees of size $n$ whose maximum protection number is at most $h$
is asymptotically given by
  \begin{equation*}
    \exp\big({-}\kappa n d^{-r^h}(1 + o(1)) + o(1)\big)
  \end{equation*}
  as $n \to \infty$, $n\equiv 1\pmod D$, and $h = \log_r{(\log_d(n))} + O(1)$, where $\kappa = \frac{w_r \lambda_1^r}{\Phi(\tau)}$ and $d = \mu^{-r} > 1$ are positive constants with $\lambda_1$ and $\mu$ defined in~\eqref{eq:lambda1} and~\eqref{eq:mu} respectively (see Lemma~\ref{lem:eta-asymp-double-exp}).
Moreover, there is a sequence of positive integers $h_n$ such that the maximum protection number of a tree with $n$ vertices is $h_n$ or $h_n+1$ with high probability (i.e., probability tending to $1$ as $n \to \infty$)
where $n\equiv 1\pmod D$.

Specifically, with $m_n = \log_r{\log_d{(n)}}$
  and $\{m_n\}$ denoting its fractional part, one can set
\begin{equation*}
h_n = \begin{cases} \lfloor m_n \rfloor & \text{if } \{m_n\} \leq \frac12, \\ \lceil m_n \rceil & \text{if } \{m_n\} > \frac12. \end{cases}
\end{equation*}
If we restrict to those values of $n$ for which $\{m_n\} \in [\varepsilon,1-\varepsilon]$, where $\varepsilon > 0$ is fixed, then with high probability $X_n$ is equal to $\lceil m_n \rceil$.
\end{maintheorem}

Note that in the setting of Theorem~\ref{thm:main_outdeg1_forbidden}, it is easy to see that there
are no trees of size $n$ if $n\not\equiv 1\pmod D$. In the setting of Theorem~\ref{thm:main_outdeg1_allowed},
we have $\gcd\{i\in\N\colon w_i\neq 0\}=1$ because $w_1\neq 0$.
Theorem~\ref{thm:main_outdeg1_allowed} is illustrated in Figure~\ref{fig:cdf-sinle-exp}, while Theorem~\ref{thm:main_outdeg1_forbidden} is illustrated in Figure~\ref{fig:cdf-double-exp}.

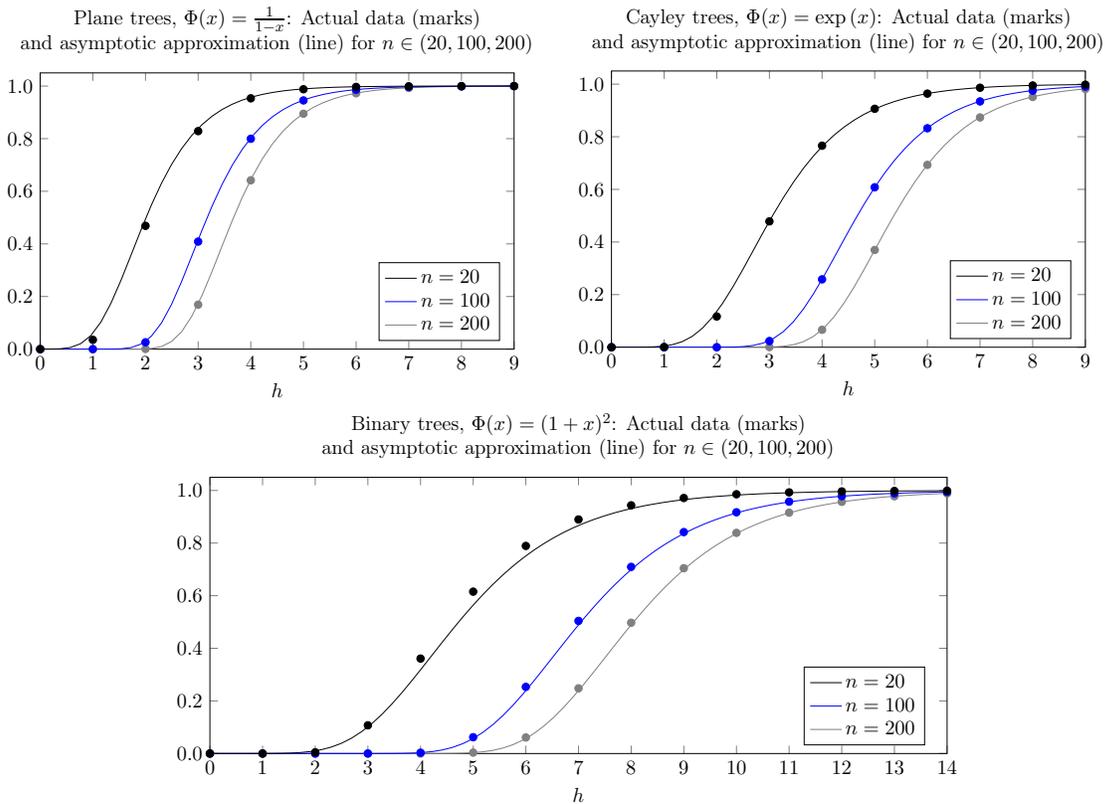
\begin{figure}[ht]
  \centering
  \begin{minipage}{0.48\textwidth}
    \begin{tikzpicture}[scale = 0.7]
      \begin{axis}[
        title style = {align=center},
        title = {Plane trees, $\Phi(x) = \frac{1}{1-x}$: Actual data (marks)\\ and asymptotic approximation (line) for $n \in (20, 100, 200)$},
        xlabel = $h$,
        xtick distance = {1},
        ytick = {0, 0.2, 0.4, 0.6, 0.8, 1},
        xmin = 0,
        xmax = 9,
        ymin = 0,
        ymax = 1.05,
        reverse legend,
        y tick label style = {/pgf/number format/.cd, fixed, fixed zerofill, precision=1, /tikz/.cd},
        x = 1cm,
        y = 5cm,
        legend pos = south east,
        legend cell align = {left},
        tick label style={/pgf/number format/1000 sep=}]
        \addplot [only marks, gray, mark = *, forget plot] coordinates{(0, 0.000000000000000)
        (1, 1.10563373688970e-14)
        (2, 0.000704327730603947)
        (3, 0.169047489664232)
        (4, 0.641811509808680)
        (5, 0.894883111927821)
        (6, 0.972549734007141)
        (7, 0.993047925987236)
        (8, 0.998252951545954)
        (9, 0.999561824191278)
        (10, 0.999890154303133)};
        \addplot [color=gray, domain = 0:19, samples = 100] {exp(-9/16*200*(0.25)^x)};\addlegendentry{$n = 200$}
        \addplot [only marks, blue, mark = *, forget plot] coordinates{(0, 0.000000000000000)
        (1, 9.79013389607174e-8)
        (2, 0.0260649999064339)
        (3, 0.408814480198590)
        (4, 0.799621796734508)
        (5, 0.945408900428163)
        (6, 0.985994088637177)
        (7, 0.996461724454492)
        (8, 0.999109599350156)
        (9, 0.999776143026773)
        (10, 0.999943730771189)
        (11, 0.999985856113744)
        (12, 0.999996444626615)
        (13, 0.999999106229246)
        (14, 0.999999775304990)
        (15, 0.999999943507860)
        (16, 0.999999985796001)
        (17, 0.999999996428407)
        (18, 0.999999999101863)
        (19, 0.999999999774133)};
        \addplot [color=blue, domain = 0:19, samples = 100] {exp(-9/16*100*(0.25)^x)};\addlegendentry{$n = 100$}
        \addplot [only marks, black, mark = *, forget plot] coordinates{(0, 0.000000000000000)
        (1, 0.0353799648910158)
        (2, 0.468658690451020)
        (3, 0.828688679815536)
        (4, 0.953368747004187)
        (5, 0.987843763384219)
        (6, 0.996862069449969)
        (7, 0.999190972794494)
        (8, 0.999791178114919)
        (9, 0.999945996725027)
        (10, 0.999986002084953)
        (11, 0.999996361888803)
        (12, 0.999999051510393)
        (13, 0.999999751832706)
        (14, 0.999999934797052)
        (15, 0.999999982785203)
        (16, 0.999999995429700)
        (17, 0.999999998781253)
        (18, 0.999999999695313)
        (19, 0.999999999847657)};
        \addplot [color=black, domain = 0:19, samples = 100] {exp(-9/16*20*(0.25)^x)};\addlegendentry{$n = 20$}
      \end{axis}
    \end{tikzpicture}
  \end{minipage}
  \begin{minipage}{0.48\textwidth}
    \begin{tikzpicture}[scale = 0.7]
      \begin{axis}[
        title style = {align=center},
        title = {Cayley trees, $\Phi(x) = \exp{(x)}$: Actual data (marks)\\ and asymptotic approximation (line) for $n \in (20, 100, 200)$},
        xlabel = $h$,
        xtick distance = {1},
        ytick = {0, 0.2, 0.4, 0.6, 0.8, 1},
        xmin = 0,
        xmax = 9,
        ymin = 0,
        ymax = 1.05,
        reverse legend,
        y tick label style = {/pgf/number format/.cd, fixed, fixed zerofill, precision=1, /tikz/.cd},
        x = 1cm,
        y = 5cm,
        legend pos = south east,
        legend cell align = {left},
        tick label style={/pgf/number format/1000 sep=}]
        \addplot [only marks, gray, mark = *, forget plot] coordinates{(0, 0.000000000000000)
        (1, 3.72352063844795e-32)
        (2, 3.48479992203085e-10)
        (3, 0.000544150205240558)
        (4, 0.0662736905151446)
        (5, 0.369669710875428)
        (6, 0.693075650202564)
        (7, 0.873525185764437)
        (8, 0.951344012432970)
        (9, 0.981769679950722)
        };
        \addplot [color=gray, domain = 0:19, samples = 100] {exp(-3.1789/exp(1)*(1-1/exp(1))*200*(1/exp(1))^x)};\addlegendentry{$n = 200$}
        \addplot [only marks, blue, mark = *, forget plot] coordinates{(0, 0.000000000000000)
        (1, 1.97925054797167e-16)
        (2, 0.0000189797575202735)
        (3, 0.0234932125719957)
        (4, 0.257959890799478)
        (5, 0.608113581959725)
        (6, 0.832306634557732)
        (7, 0.934409897706187)
        (8, 0.975229840254062)
        (9, 0.990768817275556)};
        \addplot [color=blue, domain = 0:19, samples = 100] {exp(-3.1789/exp(1)*(1-1/exp(1))*100*(1/exp(1))^x)};\addlegendentry{$n = 100$}
        \addplot [only marks, black, mark = *, forget plot] coordinates{(0, 0.000000000000000)
        (1, 0.000750801302107187)
        (2, 0.116826225117928)
        (3, 0.478374806124828)
        (4, 0.766093876694863)
        (5, 0.906339686765497)
        (6, 0.964027522371637)
        (7, 0.986385609971179)
        (8, 0.994872183438988)
        (9, 0.998070658196885)};
        \addplot [color=black, domain = 0:19, samples = 100] {exp(-3.1789/exp(1)*(1-1/exp(1))*20*(1/exp(1))^x)};\addlegendentry{$n = 20$}      
      \end{axis}
    \end{tikzpicture}
  \end{minipage}

  \begin{tikzpicture}[scale = 0.7]
    \begin{axis}[
      title style = {align=center},
      title = {Binary trees, $\Phi(x) = (1+x)^2$: Actual data (marks)\\ and asymptotic approximation (line) for $n \in (20, 100, 200)$},
      xlabel = $h$,
      xtick distance = {1},
      ytick = {0, 0.2, 0.4, 0.6, 0.8, 1},
      xmin = 0,
      xmax = 14,
      ymin = 0,
      ymax = 1.05,
      reverse legend,
      y tick label style = {/pgf/number format/.cd, fixed, fixed zerofill, precision=1, /tikz/.cd},
      x = 1cm,
      y = 5cm,
      legend pos = south east,
      legend cell align = {left},
      tick label style={/pgf/number format/1000 sep=}]
      \addplot [only marks, gray, mark = *, forget plot] coordinates{(0, 0.000000000000000)
      (1, 3.11677700934574e-88)
      (2, 1.78439182596027e-26)
      (3, 1.26679931480945e-11)
      (4, 8.57483037854609e-6)
      (5, 0.00349933272629547)
      (6, 0.0610595763539754)
      (7, 0.247854466365885)
      (8, 0.497282641317521)
      (9, 0.704455184137879)
      (10, 0.838845886516252)
      (11, 0.915635988461520)
      (12, 0.956766290066581)
      (13, 0.978085232351888)
      (14, 0.988953286770460)
      (15, 0.994447289761476)
      (16, 0.997212853871863)};
      \addplot [color=gray, domain = 0:19, samples = 100] {exp(-3.664/2*1/2*200*(0.5)^x)};\addlegendentry{$n = 200$}
      \addplot [only marks, blue, mark = *, forget plot] coordinates{(0, 0) 
      (1, 1.59325008156968e-43)
      (2, 1.90890239696768e-13)
      (3, 4.21096559837611e-6)
      (4, 0.00319933808219555)
      (5, 0.0620116374151269)
      (6, 0.253269859523321)
      (7, 0.504122336074173)
      (8, 0.709543000957086)
      (9, 0.841793671899265)
      (10, 0.917138613926652)
      (11, 0.957479766447915)
      (12, 0.978409107342060)
      (13, 0.989095201547303)
      (14, 0.994507282950451)
      (15, 0.997237089176211)
      (16, 0.998611148568838)};
      \addplot [color=blue, domain = 0:19, samples = 100] {exp(-3.664/2*1/2*100*(0.5)^x)};\addlegendentry{$n = 100$}
      \addplot [only marks, black, mark = *, forget plot] coordinates{(0, 0) (1, 2.09267723425672e-8)
      (2, 0.00473015355817857)
      (3, 0.107319597299155)
      (4, 0.360944509139098)
      (5, 0.615340184413633)
      (6, 0.789146729585558)
      (7, 0.889700743076416)
      (8, 0.943568151574927)
      (9, 0.971356982598647)
      (10, 0.985613262549932)
      (11, 0.992784197775015)
      (12, 0.996386867194422)
      (13, 0.998195609981534)
      (14, 0.999101320688521)
      (15, 0.999554678284550)
      (16, 0.999780352597492)
      };
      \addplot [color=black, domain = 0:19, samples = 100] {exp(-3.664/2*1/2*20*(0.5)^x)};\addlegendentry{$n = 20$}
    \end{axis}
  \end{tikzpicture}

  \caption{The asymptotic cumulative distribution function plotted against calculated values for plane, binary, and Cayley trees.}
  \label{fig:cdf-sinle-exp}
\end{figure}

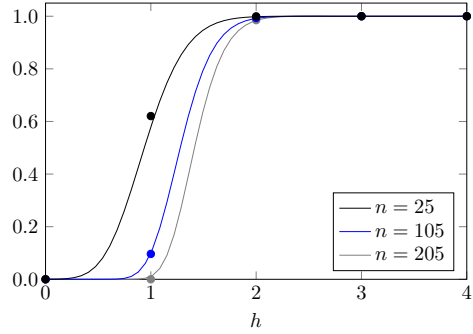
\begin{figure}[ht]
  \centering
  \begin{minipage}{0.48\textwidth}
    \begin{tikzpicture}[scale = 0.7]
      \begin{axis}[
        title style = {align=center},
        title = {Complete binary trees, $\Phi(x) = 1 + x^2$: Actual data (marks)\\ and asymptotic approximation (line) for $n \in (25, 105, 205)$},
        xlabel = $h$,
        xtick distance = {1},
        ytick = {0, 0.2, 0.4, 0.6, 0.8, 1},
        xmin = 0,
        xmax = 4,
        ymin = 0,
        ymax = 1.05,
        reverse legend,
        y tick label style = {/pgf/number format/.cd, fixed, fixed zerofill, precision=1, /tikz/.cd},
        x = 2cm,
        y = 5cm,
        legend pos = south east,
        legend cell align = {left},
        tick label style={/pgf/number format/1000 sep=}]
        \addplot [only marks, gray, mark = *, forget plot] coordinates{(0, 0.000000000000000)
        (1, 1.82640837241991e-28)
        (2, 0.231926342113497)
        (3, 1)
        (4, 1)};
        \addplot [color=gray, domain = 0:4, samples = 60] {exp(-2*205*(0.25)^(2^x))};\addlegendentry{$n = 205$}
        \addplot [only marks, blue, mark = *, forget plot] coordinates{(0, 0.000000000000000)
        (1, 7.54524917781116e-14)
        (2, 0.524373436833892)
        (3, 1)
        (4, 1)};
        \addplot [color=blue, domain = 0:4, samples = 60] {exp(-2*105*(0.25)^(2^x))};\addlegendentry{$n = 105$}
        \addplot [only marks, black, mark = *, forget plot] coordinates{(0, 0.000000000000000)
        (1, 0.00984558583158664)
        (2, 0.958944676268677)
        (3, 1)
        (4, 1)};
        \addplot [color=black, domain = 0:4, samples = 60] {exp(-2*25*(0.25)^(2^x))};\addlegendentry{$n = 25$}
      \end{axis}
    \end{tikzpicture}
  \end{minipage}
  \begin{minipage}{0.48\textwidth}
    \begin{tikzpicture}[scale = 0.7]
      \begin{axis}[
        title style = {align=center},
        title = {Riordan trees, $\Phi(x) = \frac{1}{1-x}-x$: Actual data (marks)\\ and asymptotic approximation (line) for $n \in (25, 105, 205)$},
        xlabel = $h$,
        xtick distance = {1},
        ytick = {0, 0.2, 0.4, 0.6, 0.8, 1},
        xmin = 0,
        xmax = 4,
        ymin = 0,
        ymax = 1.05,
        reverse legend,
        y tick label style = {/pgf/number format/.cd, fixed, fixed zerofill, precision=1, /tikz/.cd},
        x = 2cm,
        y = 5cm,
        legend pos = south east,
        legend cell align = {left},
        tick label style={/pgf/number format/1000 sep=}]
        \addplot [only marks, gray, mark = *, forget plot] coordinates{(0, 0.000000000000000)
        (1, 9.46076533809873e-8)
        (2, 0.985129537940785)
        (3, 0.999999870952963)
        (4, 1)
        };
        \addplot [color=gray, domain = 0:4, samples = 60] {exp(-6*205*(0.0603722)^(2^x))};\addlegendentry{$n = 205$}
        \addplot [only marks, blue, mark = *, forget plot] coordinates{(0, 0.000000000000000)
        (1, 0.0964863390158023)
        (2, 0.993266921217424)
        (3, 0.999999966602622)
        (4, 1)};
        \addplot [color=blue, domain = 0:4, samples = 60] {exp(-6*105*(0.0603722)^(2^x))};\addlegendentry{$n = 105$}
        \addplot [only marks, black, mark = *, forget plot] coordinates{(0, 0.000000000000000)
        (1, 0.620488536882679)
        (2, 0.999502293779699)
        (3, 1)
        (4, 1)};
        \addplot [color=black, domain = 0:4, samples = 60] {exp(-6*25*(0.0603722)^(2^x))};\addlegendentry{$n = 25$}      
      \end{axis}
    \end{tikzpicture}
  \end{minipage}

  \caption{The asymptotic cumulative distribution function plotted against calculated values for complete binary, and 
  Riordan trees~\cite{kim-stanley:2016:hex-trees}.}
  \label{fig:cdf-double-exp}
\end{figure}

The proof of Theorem~\ref{thm:main_outdeg1_allowed} relies on a a general distributional result provided in
\cite{Prodinger-Wagner:ta:boots}, see Theorem~\ref{thm:pro-wags-1}. For the
proof of  Theorem~\ref{thm:main_outdeg1_forbidden}, however, we will need a
variant for doubly-exponential convergence of the dominant singularities.
The statement and proof are similar to the original and we expect that this
variant will be useful in other contexts, too.

\begin{maintheorem}\label{thm:pro-wags-doubly-1}
  Let $Y_h(x) = \sum_{n \geq 0} y_{h, n}x^n$ ($h \geq 0$) be a sequence of generating functions with nonnegative coefficients 
  such that $y_{h, n}$ is nondecreasing in $h$ and (coefficientwise)
  \begin{equation*}
    \lim_{h \to \infty} Y_h(x) = Y(x) = \sum_{n \geq 0}y_n x^n, 
  \end{equation*}
  and let $X_n$ denote the sequence of random variables with support $\mathbb{N}_0$ defined by 
  \begin{equation*}
    \mathbb{P}(X_n \leq h) = \frac{y_{h, n}}{y_n}.
  \end{equation*}
  Assume that each generating function $Y_h$ has a singularity at $\rho_h \in \mathbb{R}$ such that 
  \begin{enumerate}
    \item \label{bullet:thm-doubly-1-1}$\rho_h = \rho(1+ \kappa \zeta^{r^{h}} + o(\zeta^{r^{h}}))$ as $h \to \infty$ for some constants 
    $\rho > 0$, $\kappa > 0$, $\zeta \in (0, 1)$, and $r > 1$.
    \item \label{bullet:thm-doubly-1-2}$Y_h(x)$ can be continued analytically to the domain 
    \begin{equation*}
      \{x \in \mathbb{C} : \abs{x} \geq (1+\delta)\abs{\rho_h}, \abs{\Arg(x/\rho_h - 1)} > \phi\}
    \end{equation*}
    for some fixed $\delta >0$ and $\phi \in (0, \pi/2)$, and 
    \begin{equation*}
      Y_h(x) = U_h(x) + A_h(1-x/\rho_h)^\alpha + o((1 - x/\rho_h)^\alpha)
    \end{equation*}
    holds within this domain, uniformly in $h$, where $U_h(x)$ is analytic and uniformly bounded in $h$ within the aforementioned region, 
    $\alpha \in \mathbb{R} \setminus \mathbb{N}_0$, and $A_h$ is a constant dependent on $h$ such that $\lim_{h \to \infty}A_h = A\neq 0$. Finally, 
    \begin{equation*}
      Y(x) = U(x) + A (1-x/\rho)^\alpha + o((1-x/\rho)^\alpha)
    \end{equation*}
    in the region
    \begin{equation*}
      \{x \in \mathbb{C} : \abs{x} \geq (1+\delta)\abs{\rho}, \abs{\Arg(x/\rho - 1)} > \phi\}
    \end{equation*}
    for a function $U(x)$ that is analytic within this region. 
  \end{enumerate}
  Then the asymptotic formula
  \begin{equation*}
  \mathbb{P}(X_n \leq h) = \frac{y_{h, n}}{y_n} = \exp{\big({-}\kappa n\zeta^{r^h}(1+o(1)) + o(1)\big)}
  \end{equation*}
  holds as $n \to \infty$ and $h = \log_r{(\log_d{(n)})} + O(1)$, where $d = \zeta^{-1}$.
\end{maintheorem}

Note that here we have $\rho_h = \rho(1+ \kappa \zeta^{r^{h}} + o(\zeta^{r^{h}}))$, while in 
Theorem~\ref{thm:pro-wags-1} we have the exponential case $\rho_h = \rho (1+ \kappa \zeta^h + o(\zeta^h))$.

In the next theorem, we show that the consequences of this distributional result are quite drastic.

\begin{maintheorem}\label{thm:doubly-exponential-two-points}
  Assume the conditions of Theorem~\ref{thm:pro-wags-doubly-1}. There is a
  sequence of nonnegative integers $h_n$ such that $X_n$ is equal to $h_n$ or
  $h_n+1$ with high probability. Specifically, with $m_n = \log_r{\log_d{(n)}}$
  and $\{m_n\}$ denoting its fractional part, one can set
\begin{equation*}
h_n = \begin{cases} \lfloor m_n \rfloor & \text{if } \{m_n\} \leq \frac12, \\ \lceil m_n \rceil & \text{if } \{m_n\} > \frac12. \end{cases}
\end{equation*}
If we restrict to those values of $n$ for which $\{m_n\} \in [\varepsilon,1-\varepsilon]$, where $\varepsilon > 0$ is fixed, then with high probability $X_n$ is equal to $\lceil m_n \rceil$.
\end{maintheorem}

\section{Preliminaries}

\subsection{Basic facts about simply generated trees}\label{sec:sg_trees_basic}

For our purposes, we will make the following typical technical assumptions:
first, we assume without loss of generality that $w_0 = 1$ or equivalently
$\Phi(0) = 1$. In other words, leaves have an associated weight of $1$, which
can be achieved by means of a normalising factor if necessary. Moreover, to
avoid trivial cases in which the only possible trees are paths, we assume that
$w_j > 0$ for at least one $j \geq 2$. Finally, we assume that there is a
positive real number $\tau$, less than the radius of convergence of $\Phi$,
such that $\Phi(\tau) = \tau\Phi'(\tau)$. As mentioned earlier, this is
equivalent to the offspring distribution having exponential moments.

It is well known (see e.g.~\cite[Section 3.1.4]{Drmota:2009:random}) that if such a $\tau$ exists, it is unique, and the radius of convergence $\rho$ of $Y$ can be expressed as
\begin{equation}\label{eq:fundamental-connection-between-greek-variables}
    \rho = \tau/\Phi(\tau) = 1/\Phi'(\tau),
\end{equation}
which is equivalent to $\rho$ and $\tau$ satisfying the simultaneous equations $y = x \Phi(y)$ and $1 = x \Phi'(y)$ (which essentially mean that the implicit function theorem fails at the point $(\rho,\tau)$). 
Moreover, $Y$ has a square root singularity at $\rho$ with $\tau = Y(\rho)$, with a singular expansion of the form
\begin{equation}\label{eq:singular_expansion_Y}
Y(x) = \tau + a \Bigl(1-\frac{x}{\rho}\Bigr)^{1/2} + b\Bigl(1-\frac{x}{\rho}\Bigr) + c \Bigl(1-\frac{x}{\rho}\Bigr)^{3/2} + O\Bigl((\rho-x)^2\Bigr).
\end{equation}
The coefficients $a,b,c$ can be expressed in terms of $\Phi$ and $\tau$. In particular, we have
\begin{equation*}
a = - \Big( \frac{2\Phi(\tau)}{\Phi''(\tau)} \Big)^{1/2}.
\end{equation*}
In fact, there is a full Newton--Puiseux expansion in powers of
$(1-x/\rho)^{1/2}$. If the weight sequence is \emph{aperiodic}, i.\,e.,
$\gcd \{j \colon w_j \neq 0\} = 1$, then $\rho$ is the only singularity on the
circle of convergence of $Y$, and for sufficiently small $\varepsilon > 0$
there are no solutions to the simultaneous equations $y = x \Phi(y)$ and
$1 = x \Phi'(y)$ with $\abs{x} \leq \rho + \varepsilon$ and
$\abs{y} \leq \tau + \varepsilon$ other than $(x,y) = (\rho,\tau)$. Otherwise,
if this $\gcd$ is equal to $D$, there are $D$ singularities at
$\rho e^{2k\pi i}{D}$ ($i \in \{0,1,\ldots,D-1\}$), all with the same singular
behaviour. In the following, we assume for technical simplicity that the weight
sequence is indeed aperiodic, but the proofs are readily adapted to the
periodic setting, see Remarks~\ref{remark:periodic-case-1} and~\ref{remark:periodic-case-2}.

By means of singularity analysis \cite[Chapter VI]{Flajolet-Sedgewick:ta:analy}, the singular expansion~\eqref{eq:singular_expansion_Y} yields an asymptotic formula for the coefficients of $Y$: we have
\begin{equation*}
y_n = [x^n] Y(x) \sim \frac{-a}{2\sqrt{\pi}} n^{-3/2} \rho^{-n}.
\end{equation*}
If the weight sequence corresponds to a probability distribution, then $y_n$ is
the probability that an \emph{unconditioned} Bienaym\'e--Galton--Watson tree
has exactly $n$ vertices when the process ends. For other classes such as plane
trees or binary trees, $y_n$ represents the number of $n$-vertex trees in the
respective class.

\subsection{A general distributional result}\label{sec:general_distributional}

The discrete double-exponential distribution in
Theorem~\ref{thm:main_outdeg1_allowed} has been observed in many other
combinatorial instances, for example the longest run of zeros in a random
$0$-$1$-string, the longest horizontal segment in Motzkin paths or the maximum
outdegree in plane trees. This can often be traced back to the behaviour of the
singularities of associated generating functions. The following general 
result~\cite{Prodinger-Wagner:ta:boots}, similar to 
Theorem~\ref{thm:pro-wags-doubly-1} but with an exponential instead of doubly-exponential 
rate of convergence of the dominant singularity, will be a key tool for us.

\begin{theorem}[{see~\cite[Theorem 1]{Prodinger-Wagner:ta:boots}}]\label{thm:pro-wags-1}
Let $Y_h(x) = \sum_{n \geq 0}y_{h, n}x^n$ $(h \geq 0)$ be a sequence of generating functions with nonnegative coefficients 
such that $y_{h, n}$ is nondecreasing in $h$ and (coefficientwise)
\begin{equation*}
  \lim_{h \to \infty} Y_h(x) = Y(x) = \sum_{n \geq 0}y_n x^n,
\end{equation*}
and let $X_n$ denote the sequence of random variables with support $\N_0$ defined by 
\begin{equation}\label{eq:definition-of-Xn}
  \mathbb{P}(X_n \leq h) = \frac{y_{h, n}}{y_n}.
\end{equation}
Assume, moreover, that each generating function $Y_h$ has a singularity $\rho_h \in \R$, such that 
\begin{enumerate}
  \item $\rho_h = \rho (1+ \kappa \zeta^h + o(\zeta^h))$ as $h \to \infty$ for some constants $\rho > 0$, $\kappa > 0$ and $\zeta \in (0, 1)$. \label{bullet:thm-1-1}
  \item $Y_h(x)$ can be continued analytically to the domain \label{bullet:thm-1-2}
  \begin{equation}\label{eq:delta-domain}
    \{x \in \C: \abs{x} \leq (1 + \delta)\abs{\rho_h}, \abs{\Arg(x/\rho_h - 1)} > \phi\}
  \end{equation}
  for some fixed $\delta > 0$ and $\phi \in (0, \pi/2)$, and 
  \begin{equation*}
    Y_h(x) = U_h(x) + A_h(1-x/\rho_h)^\alpha + o((1-x/\rho_h)^\alpha)
  \end{equation*}
  holds within this domain, uniformly in $h$, where $U_h(x)$ is analytic and uniformly bounded in $h$ within the aforementioned region, 
  $\alpha \in \R \setminus \N_0$, and $A_h$ is a constant depending on $h$ such that $\lim_{h \to \infty}A_h = A\neq 0$. Finally, 
  \begin{equation*}
    Y(x) = U(x) + A(1-x/\rho)^\alpha + o((1-x/\rho)^\alpha)
  \end{equation*}
  in the region 
  \begin{equation*}
    \{x \in \C: \abs{x} \leq (1 + \delta)\abs{\rho}, \abs{\Arg(x/\rho - 1)} > \phi\}
  \end{equation*}
  for a function $U(x)$ that is analytic within this region. 
\end{enumerate}
Then the asymptotic formula 
\begin{equation*}
    \mathbb{P}(X_n \leq h) = \frac{y_{h, n}}{y_n} = \exp{(-\kappa n\zeta^h)}(1 + o(1))
\end{equation*}
holds as $n \to \infty$ and $h = \log_d(n) + O(1)$, where $d = \zeta^{-1}$. Hence the shifted random variable 
$X_n - \log_d(n)$ converges weakly to a limiting distribution if $n$ runs through a subset of the positive integers such that the 
fractional part $\{\log_d(n)\}$ of $\log_d(n)$ converges.
\end{theorem}

As we will see, the conditions of this theorem hold for the random variable $X_n$ given by the maximum protection number of a random $n$-vertex tree from a simply generated family that satisfies our technical assumptions. Under slightly stronger assumptions, which also hold in our case, one has the following theorem on the expected value of the random variable $X_n$.

\begin{theorem}[{see~\cite[Theorem 2]{Prodinger-Wagner:ta:boots}}]\label{thm:pro-wags-2}
  In the setting of Theorem~\ref{thm:pro-wags-1}, assume additionally that 
  \begin{enumerate}
    \item There exists a constant $K$ such that $y_{h, n} = y_n$ for $h > Kn$, \label{bullet:thm-2-1}
    \item $\sum_{h \geq 0} \abs{A - A_h} < \infty$, \label{bullet:thm-2-2}
    \item the asymptotic expansions of $Y_h$ and $Y$ around their singularities are given by \label{bullet:thm-2-3}
    \begin{equation*}
      Y_h(x) = U_h(x) + A_h(1 - x/\rho_h)^\alpha + B_h(1 - x/\rho_h)^{\alpha+1} + o((1 - x/\rho_h)^{\alpha+1}),
    \end{equation*}
    uniformly in $h$, and 
    \begin{equation*}
      Y(x) = U(x) + A(1 - x/\rho_h)^\alpha + B(1 - x/\rho_h)^{\alpha+1} + o((1 - x/\rho_h)^{\alpha+1}),
    \end{equation*}
    respectively, such that $\lim_{h\to \infty} B_h = B$. 
  \end{enumerate}
  Then the mean of $X_n$ satisfies 
  \begin{equation*}
    \mathbb{E}(X_n) = \log_d(n) + \log_d(\kappa) + \frac{\gamma}{\log(d)} + \frac{1}{2} + \psi_d(\log_d(\kappa n)) + o(1),
  \end{equation*}
  where $\gamma$ denotes the Euler--Mascheroni constant and $\psi_d$ is given by~\eqref{eq:psib}.
\end{theorem}

\subsection{A system of functional equations}\label{sec:system-funct-eqn}

As a first step of our analysis, we consider a number of auxiliary generating
functions and derive a system of functional equations that is satisfied by
these generating functions.  The family of simply generated trees and the
associated weight generating function $\Phi$ are regarded fixed throughout. Let
$h$ be a positive integer and $k$ an integer with $0 \leq k \leq h$. Consider
trees with the following two properties:
\begin{enumerate}[label=P\arabic*.,ref=P\arabic*]
\item\label{property1} No vertex has a protection number greater than $h$.
\item\label{property2} The root is $k$-protected (but also has protection number at most $h$).
\end{enumerate}
Let $Y_{h, k}(x)$ be the associated generating function, where $x$ marks the
number of vertices. Note in particular that when $k = 0$, we obtain the
generating function for trees where the maximum protection number is at most
$h$. Hence we can express the probability that the maximum protection number of
a random $n$-vertex tree (from our simply generated family) is at most $h$ as
the quotient
\begin{equation*}
\frac{[x^n] Y_{h,0}(x)}{[x^n] Y(x)}.
\end{equation*}
This is precisely the form of~\eqref{eq:definition-of-Xn}, and indeed our
general strategy will be to show that the generating functions $Y_{h,0}$
satisfy the technical conditions of Theorem~\ref{thm:pro-wags-1}. Compared to
the examples given in~\cite{Prodinger-Wagner:ta:boots}, this will be a rather
lengthy technical task. However, we believe that the general method, in which a
sequence of functional equations is shown to converge uniformly in a suitable
region, is also potentially applicable to other instances and therefore
interesting in its own right.

Let us now derive a system of functional equations, using the standard
decomposition of a rooted tree into the root and its branches. Clearly, if a
tree has property~\ref{property1}, then this must also be the case for all its
branches. Moreover, property~\ref{property2} is satisfied for $k > 0$ if and
only if the root of each of the branches is at least $(k-1)$-protected, but not
all of them are $h$-protected (as this would make the root
$(h+1)$-protected). Thus, for $1 \leq k \leq h$, we have
\begin{equation}\label{eq:func-eq-2}
Y_{h, k}(x) = x\Phi(Y_{h, k-1}(x)) - x\Phi(Y_{h, h}(x)).
\end{equation}
Note that the only case in which the root is only $0$-protected is when the root is the only vertex. Hence we have
\begin{equation}\label{eq:func-eq-1}
  Y_{h, 0}(x) = Y_{h, 1}(x) + x.
\end{equation}

The analytic properties of the system of functional equations given
by~\eqref{eq:func-eq-2} and~\eqref{eq:func-eq-1} will be studied in the
following section, culminating in
Proposition~\ref{prop:analytic_properties_of_Yh}, which shows that
Theorem~\ref{thm:pro-wags-1} is indeed applicable to our problem.

\section{Analysis of the functional equations}\label{sec:contraction}

\subsection{Contractions and implicit equations}

This section is devoted to a detailed analysis of the generating functions
$Y_{h,k}$ that satisfy the system of equations given by~\eqref{eq:func-eq-2}
and~\eqref{eq:func-eq-1}. The first step will be to reduce it to a single
implicit equation satisfied by $Y_{h,1}$ that is then shown to converge to the
functional equation~\eqref{eq:Y_functional_equation} in a sense that will be
made precise. This is then used to infer information on the region of
analyticity of $Y_{h,1}$ as well as its behaviour around the dominant
singularity, which is also shown to converge to the dominant singularity of
$Y$. This information is collected in
Proposition~\ref{prop:analytic_properties_of_Yh} at the end of the section.

In the following, we will prove various statements for sufficiently small
$\varepsilon>0$. In several, but finitely many, steps it might be necessary to
decrease $\varepsilon$; we tacitly assume that $\varepsilon$ is always small
enough to ensure validity of all statements up to the given point. In order to
avoid ambiguities, we will always assume that $\varepsilon<1$. Let us remark
that $\varepsilon$ and other constants as well as all implied $O$-constants
that occur in this section depend on the specific simply generated family of
trees (in particular the weight generating function $\Phi$ and therefore $\rho$ and $\tau$), but nothing else.

Recall that $\rho$ is the dominant singularity of the generating function $Y$
of our simply generated family of trees. Moreover, $\tau = Y(\rho)$ is
characterised by the equation $\tau \Phi'(\tau) = \Phi(\tau)$ (see \eqref{eq:fundamental-connection-between-greek-variables}) and satisfies
$\tau = \rho \Phi(\tau)$. Since $\Phi$ is increasing and $\Phi(0) = 1$, we also
have $\tau = \rho\Phi(\tau) > \rho\Phi(0) = \rho$.

Let us write $D_{\delta}(w) := \{ z \in \C \,: \, \abs{z-w} < \delta\}$ for open disks. For $\varepsilon>0$, we define
\begin{align*}
  \Xi_\varepsilon^{(1)}&\coloneqq D_{\rho+\varepsilon}(0),\\
  \Xi_\varepsilon^{(2)}&\coloneqq D_{\tau - \rho+\varepsilon}(0),\\
  \Xi_\varepsilon^{(3)}&\coloneqq D_{\varepsilon}(0).
\end{align*}
For $1\le j<k\le 3$, we set $\Xi^{(j, k)}_\varepsilon\coloneqq
\Xi^{(j)}_\varepsilon\times \Xi^{(k)}_\varepsilon$, and we also set
$\Xi_{\varepsilon}\coloneqq \Xi^{(1, 2,
  3)}_{\varepsilon}\coloneqq \Xi^{(1)}_\varepsilon\times
\Xi^{(2)}_\varepsilon\times \Xi^{(3)}_\varepsilon$.
As $\tau$ is less than the radius of convergence of $\Phi$ by our assumptions, we
may choose $\varepsilon>0$ sufficiently small such that $\tau+2\varepsilon$ is still smaller than the radius of convergence of $\Phi$.

Consider the function defined by $f_{x, z}(y) = x(\Phi(y)-\Phi(z))$. We can rewrite the functional equation~\eqref{eq:func-eq-2} in terms of this function as
\begin{equation}\label{eq:functional-equation-via-F}
  Y_{h, k}(x) = f_{x, Y_{h, h}(x)}(Y_{h, k-1}(x))
\end{equation}
for $1\le k\le h$. For $j\ge 0$, we denote
the $j$th iterate of $f_{x, z}$ by $f^{(j)}_{x, z}$, i.\,e., $f^{(0)}_{x, z}(y) = y$
and $f^{(j+1)}_{x, z}(y)=f_{x, z}^{(j)}(f_{x, z}(y))$ for $j\ge
0$. Iterating~\eqref{eq:functional-equation-via-F} then yields
\begin{equation*}
  Y_{h, k}(x) = f_{x, Y_{h,h}(x)}(Y_{h, k-1}(x))=\cdots = f_{x, Y_{h,
      h}(x)}^{(k-1)}(Y_{h, 1}(x))
\end{equation*}
for $1\le k\le h$ and therefore
\begin{equation}\label{eq:short-system-1}
  Y_{h, h}(x) = f^{(h-1)}_{x, Y_{h, h}(x)}(Y_{h, 1}(x)).
\end{equation}

Plugging~\eqref{eq:func-eq-1} into~\eqref{eq:func-eq-2} for
$k=1$ yields
\begin{equation}\label{eq:short-system-2}
  Y_{h, 1}(x) = x\big(\Phi(Y_{h, 1}(x)+x)-\Phi(Y_{h, h}(x))\big).
\end{equation}
This means that~\eqref{eq:short-system-1} and~\eqref{eq:short-system-2} are a
system of two functional equations for $Y_{h, 1}(x)$ and $Y_{h, h}(x)$. We
intend to solve~\eqref{eq:short-system-1} for $Y_{h, h}(x)$ and then plug the
solution into~\eqref{eq:short-system-2}. As a first step towards this goal, we show that $f_{x,z}$ represents a contraction on a suitable region.

\begin{lemma}\label{lemma:contraction-preparation}
  For sufficiently small $\varepsilon>0$, we have $\abs{f_{x,z}(y)}<\tau-\rho$
  for all $(x, y, z)\in\Xi_{\varepsilon}$.
\end{lemma}
\begin{proof}
  By the triangle inequality, definition of $\Xi_{\varepsilon}$,
  non-negativity of the coefficients of $\Phi$, and $\Phi(0)=1$, we have
  \begin{align*}
    \abs{f_{x,z}(y)}&=\abs{x\bigl((\Phi(y)-1)-(\Phi(z)-1)\bigr)}\\
                    &\le (\rho+\varepsilon)(\abs{\Phi(y)-1}+\abs{\Phi(z)-1})\\
                    &\le (\rho+\varepsilon)((\Phi(\abs{y})-1)+(\Phi(\abs{z})-1))\\
    &\le (\rho+\varepsilon)(\Phi(\tau-\rho+\varepsilon)-1+\Phi(\varepsilon)-1).
  \end{align*}
  For $\varepsilon\to 0$, the upper bound converges to
  $\rho\Phi(\tau-\rho)-\rho$ because we are assuming that $\Phi(0)=1$. As
  $\rho\Phi(\tau-\rho)-\rho < \rho\Phi(\tau)-\rho=\tau-\rho$
  by~\eqref{eq:fundamental-connection-between-greek-variables}, the assertion of
  the lemma holds for sufficiently small $\varepsilon>0$.
\end{proof}

\begin{lemma}\label{lemma:f-prime-less-than-1}
  For sufficiently small $\varepsilon > 0$ and $(x, y, z) \in \Xi_{\varepsilon}$, we have $\abs{f'_{x, z}(y)} = \abs{x \Phi'(y)} \leq \lambda$ for some constant $\lambda < 1$. 
\end{lemma}
\begin{proof}
  For any triple $(x, y, z)\in \Xi_{\varepsilon}$,
  \begin{equation*}
    \abs{f'_{x, z}(y)}=\abs{x\Phi'(y)}\le (\rho+\varepsilon)\Phi'(\tau-\rho+\varepsilon).
  \end{equation*}
  For $\varepsilon\to 0$, the upper bound converges to $\rho\Phi'(\tau-\rho)$,
  which is less than $\rho\Phi'(\tau)=1$
  (by~\eqref{eq:fundamental-connection-between-greek-variables}).
\end{proof}
For the remainder of this section, $\lambda$ will be defined as in
Lemma~\ref{lemma:f-prime-less-than-1}.

\begin{lemma}\label{lemma:contraction}
  For sufficiently small $\varepsilon>0$ and $(x,
z)\in\Xi_{\varepsilon}^{(1, 3)}$, $f_{x,z}$ maps $\Xi_{\varepsilon}^{(2)}$ to
itself and is a contraction with Lipschitz constant $\lambda$.
\end{lemma}
\begin{proof}
  The fact that $f_{x,z}$ maps $\Xi_{\varepsilon}^{(2)}$ to
  itself for sufficiently small $\varepsilon>0$ is a direct
  consequence of Lemma~\ref{lemma:contraction-preparation}.

   Making use of Lemma~\ref{lemma:f-prime-less-than-1}, the contraction property now follows by a standard argument:
  For $y_1$, $y_2\in \Xi_{\varepsilon}^{(2)}$, we have
  \begin{equation*}
    \abs{f_{x, z}(y_2)-f_{x, z}(y_1)}\le \int_{[y_1, y_2]}\abs{f'_{x,
        z}(y)}\,\abs{dy}\le \lambda \abs{y_2-y_1}.\qedhere
  \end{equation*}
\end{proof}

For sufficiently small $\varepsilon$ and $(x, z)\in\Xi_{\varepsilon}^{(1, 3)}$,
Banach's fixed point theorem together with Lemma~\ref{lemma:contraction}
implies that $f_{x,z}$ has a unique fixed point in $\Xi_{\varepsilon}^{(2)}$.
This fixed point will be denoted by $g(x, z)$, i.\,e.,
\begin{equation}\label{eq:implicit-equation-for-g}
  g(x, z)=f_{x, z}(g(x, z)) = x(\Phi(g(x, z))-\Phi(z)).
\end{equation}
If we plug in $0$ for $z$, we see that~\eqref{eq:implicit-equation-for-g} holds for $g(x,0) =  0$, so uniqueness of the fixed point implies that
\begin{equation}\label{eq:g-when-z-is-zero}
  g(x, 0)=0
\end{equation}
for $x\in\Xi_{\varepsilon}^{(1)}$.

\begin{lemma}\label{lemma:g-is-analytic}
  For sufficiently small $\varepsilon>0$, $g\colon \Xi_{\varepsilon}^{(1,
    3)}\to\Xi_\varepsilon^{(2)} $ is an analytic
  function, and $\frac{\partial}{\partial z} g(x,z)$ is bounded.
\end{lemma}
\begin{proof}
   Note that using Lemma~\ref{lemma:f-prime-less-than-1}, we have that $\abs{\frac{\partial}{\partial y}(y-f_{x, z}(y))}=\abs{1-f_{x,z}'(y)} \geq 1 - \abs{f_{x,z}'(y)} \geq 1 - \lambda$ 
   is bounded away from zero for sufficiently small $\varepsilon>0$ and $(x, y, z)\in\Xi_{\varepsilon}$.
  Thus the analytic implicit function theorem shows that $g$ as defined
  by~\eqref{eq:implicit-equation-for-g} is analytic and has bounded partial derivative $\frac{\partial}{\partial z} g(x,z)$ on $\Xi_{\varepsilon}^{(1,
    3)}$ for sufficiently small $\varepsilon>0$.
\end{proof}

We now intend to solve~\eqref{eq:short-system-1} for $Y_{h, h}(x)$. Therefore, we
consider the equation
\begin{equation}\label{eq:implicit-equation-1}
  z=f_{x, z}^{(h-1)}(y)
\end{equation}
and attempt to solve it for $z$. For large $h$, $f_{x, z}^{(h-1)}(y)$ will be
close to the fixed point $g(x, z)$ of $f_{x, z}$
by the Banach fixed point theorem.

Therefore, we define
$\Lambda_h$ as the difference between the two: $\Lambda_h(x, y, z)\coloneqq
f_{x,z}^{(h-1)}(y)-g(x, z)$. So~\eqref{eq:implicit-equation-1} can be
rewritten as
\begin{equation}\label{eq:implicit-equation-2}
  z=g(x, z)+ \Lambda_h(x, y, z).
\end{equation}

We first establish bounds on $\Lambda_h$.

\begin{lemma}\label{lem:Lambda-small}
  For sufficiently small $\varepsilon>0$,
  \begin{align}
    \Lambda_h(x, y, z)&=O(\lambda^h)\label{eq:bound-for-Lambda} \text{ and }\\
    \frac{\partial}{\partial z}\Lambda_h(x, y, z)&=O(\lambda^h)\label{eq:bound-for-derivative-of-Lambda}
  \end{align}
  hold uniformly for $(x, y, z)\in\Xi_{\varepsilon}$.
\end{lemma}
\begin{proof}
  Since $g$ is defined as the fixed point of $f_{x, z}$ and $f_{x, z}$ is a contraction with Lipschitz constant $\lambda$, we have
  \begin{equation*}
    \abs{\Lambda_h(x, y, z)} = \abs{f_{x, z}^{(h-1)}(y) - f_{x, z}^{(h-1)}(g(x, z))} \le \lambda^{h-1} \abs{y-g(x, z)}=O(\lambda^h)
  \end{equation*}
  for $(x, y, z)\in\Xi_\varepsilon$, so we have shown~\eqref{eq:bound-for-Lambda}.

  For $(x, y, z)\in\Xi_{\varepsilon/3}$, Cauchy's integral formula yields
  \begin{equation*}
    \frac{\partial}{\partial z}\Lambda_h(x, y, z)=\frac{1}{2\pi i}\oint_{\abs{\zeta-z}=\varepsilon/3}\frac{\Lambda_h(x, y, \zeta)}{(\zeta-z)^2}\,d\zeta.
  \end{equation*}
  By~\eqref{eq:bound-for-Lambda}, we can bound the integral by $O(\lambda^h)$. Thus replacing $\varepsilon$ by $\varepsilon/3$ yields~\eqref{eq:bound-for-derivative-of-Lambda}.
\end{proof}

In order to apply the analytic implicit function theorem to the implicit equation~\eqref{eq:short-system-1} for $Y_{h,h}$, we will need to show that the derivative
of the difference of the two sides of~\eqref{eq:implicit-equation-2} with respect to $z$ is nonzero.
The derivative of the second summand on the right-hand side of~\eqref{eq:implicit-equation-2} is
small by~\eqref{eq:bound-for-derivative-of-Lambda}, so we first consider the
remaining part of the equation.

\begin{lemma}\label{lemma:derivative-z-minus-g-x-z-bounded-away}
  There is a $\delta>0$ such that for sufficiently small $\varepsilon>0$, we
  have
  \begin{equation}\label{eq:derivative-estimate}
    \abs*{\frac{\partial}{\partial z}(z-g(x, z))}>\delta
  \end{equation}
  for $(x, z)\in\Xi_{\varepsilon}^{(1, 3)}$.
\end{lemma}
\begin{proof}
  To compute $\frac{\partial}{\partial z} g(x, z)$, we
  differentiate~\eqref{eq:implicit-equation-for-g} with respect to $z$ and
  obtain
  \begin{equation*}
	\frac{\partial}{\partial z} g(x, z) = x\Phi'(g(x,z))\frac{\partial}{\partial z} g(x, z) - x\Phi'(z),
\end{equation*}
  which leads to
\begin{equation*}
\frac{\partial}{\partial z} g(x, z) = - \frac{x\Phi'(z)}{1-x\Phi'(g(x,
      z))}.
\end{equation*}
  Note that the denominator is nonzero for $(x, z)\in\Xi_{\varepsilon}^{(1,
    3)}$ by Lemma~\ref{lemma:f-prime-less-than-1}.
  We obtain
  \begin{equation}\label{eq:derivative-implicit-equation-2}
    \abs*{\frac{\partial}{\partial z}(z-g(x, z))} = \abs*{\frac{1+x(\Phi'(z)-\Phi'(g(x,
      z)))}{1-x\Phi'(g(x, z))}}\ge \frac{1-(\rho+\varepsilon)\abs{\Phi'(z)-\Phi'(g(x,
      z))}}{1+(\rho+\varepsilon)\abs{\Phi'(g(x, z))}}.
  \end{equation}

  By Lemma~\ref{lemma:g-is-analytic}, $\frac{\partial g(x, z)}{\partial z}$ is analytic and 
bounded for $(x,  z)\in\Xi_\varepsilon^{(1, 3)}$, and by~\eqref{eq:g-when-z-is-zero}, it follows that 
  \begin{equation*}
    g(x, z)=g(x, z)-g(x, 0)=\int_{[0, z]}\frac{\partial g(x, \zeta)}{\partial\zeta}\,d\zeta=O(\abs{z})=O(\varepsilon)
  \end{equation*}
  for $\varepsilon\to 0$, uniformly in $x$. Therefore, we have
  \begin{equation*}
\Phi'(z)-\Phi'(g(x, z)) = (\Phi'(z)-\Phi'(0))-(\Phi'(g(x,z))-\Phi'(0)) = O(\varepsilon)
  \end{equation*}
  and $\abs{\Phi'(g(x, z))}= \Phi'(0)+O(\varepsilon)$ for $\varepsilon\to 0$.
  So~\eqref{eq:derivative-implicit-equation-2} yields
  \begin{equation*}
    \abs*{\frac{\partial}{\partial z}(z-g(x,
      z))}\ge \frac{1-(\rho+\varepsilon)O(\varepsilon)}{1+(\rho+\varepsilon)(\Phi'(0)+O(\varepsilon))}
    = \frac{1}{1+\rho\Phi'(0)}+O(\varepsilon)
  \end{equation*}
  for $\varepsilon\to 0$.  Setting
  $\delta\coloneqq \frac12 \frac{1}{1+\rho\Phi'(0)}$ and choosing $\varepsilon$
  small enough yields the result.
\end{proof}

We need bounds for $z$ such that
we remain in the region where our previous results hold. In
fact,~\eqref{eq:g-when-z-is-zero} shows that $z=0$ would be a solution when the summand $\Lambda_h$
(which is $O(\lambda^h)$) is removed from the implicit equation, so we expect that
the summand $\Lambda_h$ does not perturb $z$ too much. This is shown in the
following lemma.

\begin{lemma}\label{lem:z-is-small}
  Let  $\varepsilon>0$ be sufficiently small and $(x, y, z)\in\Xi_\varepsilon$
  such that~\eqref{eq:implicit-equation-2} holds. Then
  \begin{equation}\label{eq:z-is-small}
    z=O(\lambda^h).
  \end{equation}
\end{lemma}

\begin{proof}
  In view of~\eqref{eq:implicit-equation-2}
  and~\eqref{eq:bound-for-Lambda}, we have
  \begin{equation}\label{eq:z-is-small-1}
    g(x, z) - z = O(\lambda^h).
  \end{equation}
  By definition, $g(x, z)\in\Xi_{\varepsilon}^{(2)}$. The implicit equation~\eqref{eq:implicit-equation-for-g} for $g(x,z)$ and~\eqref{eq:z-is-small-1} imply
  \begin{equation*}
    g(x, z) = x(\Phi(g(x, z)) - \Phi(z))=x \int_{[z, g(x,
      z)]}\Phi'(\zeta)\,d\zeta=O(\abs{g(x, z)-z})=O(\lambda^h).
  \end{equation*}
  Inserting this into~\eqref{eq:z-is-small-1} leads to~\eqref{eq:z-is-small}.
\end{proof}

\begin{lemma}\label{lem:q_h-unique}
  There exists an $\varepsilon>0$ such that for sufficiently large $h$, there is a unique analytic function $q_h\colon \Xi^{(1,
    2)}_\varepsilon\to \C$ such that
  \begin{equation}\label{eq:functional-equation-for-q}
    q_h(x, y)=f_{x, q_h(x,y)}^{(h-1)}(y)
  \end{equation}
  and $q_h(x, 0)=0$ for $(x, y)\in\Xi_{\varepsilon}^{(1, 2)}$; furthermore,
  $q_h(x, y)=O(\lambda^h)$ holds uniformly in $x$ and $y$.
\end{lemma}
\begin{proof}
  We
  choose $h$ sufficiently large such that
  \eqref{eq:bound-for-derivative-of-Lambda} implies
  \begin{equation}\label{eq:bound-for-derivative-of-Lambda-for-large-h}
    \abs[\Big]{\frac{\partial}{\partial z}\Lambda_h(x, y, z)}\le \frac{\delta}{2}
  \end{equation}
  for $(x, y, z)\in \Xi_\varepsilon$, where $\delta$ is taken as in
  Lemma~\ref{lemma:derivative-z-minus-g-x-z-bounded-away}, and such that~\eqref{eq:z-is-small} implies
  \begin{equation}\label{eq:bound-for-z-for-large-h}
    \abs{z}\le \frac{\varepsilon}{2}
  \end{equation}
  for all $(x, y, z)\in \Xi_\varepsilon$ for which~\eqref{eq:implicit-equation-2} holds.

  By definition of $f$, we have $f_{x, 0}(0) = 0$ and therefore
  $f_{x, 0}^{(h-1)}(0) = 0$ for every $x \in
  \Xi_\varepsilon^{(1)}$, so $z=0$ is a solution of~\eqref{eq:implicit-equation-1} for $y=0$.
  By~\eqref{eq:derivative-estimate}
  and~\eqref{eq:bound-for-derivative-of-Lambda-for-large-h}, we have
  \begin{equation}\label{eq:derivative-does-not-vanish}
    \frac{\partial}{\partial z} (f_{x,z}^{(h-1)}(y) - z)\neq 0
  \end{equation}
  for $(x,y,z) \in \Xi_\varepsilon$.
  The analytic implicit function theorem thus
  implies that, for every $x \in \Xi_\varepsilon^{(1)}$, there is an
  analytic function $q_h$ defined in a neighbourhood of $(x, 0)$ such
  that~\eqref{eq:functional-equation-for-q} holds there and such that
  $q_h(x, 0)=0$. Next we show that this extends to the whole region
  $\Xi_\varepsilon^{(1,2)}$.

  For $x_0 \in \Xi_\varepsilon^{(1)}$, let $r(x_0)$ be the supremum of all
  $r < \tau-\rho+\varepsilon$ for which there is an analytic extension of
  $y\mapsto q_h(x_0,y)$ from the open disk $D_r(0)$ to $\Xi_\varepsilon^{(3)}$. Suppose
  for contradiction that $r(x_0) < \tau - \rho + \varepsilon$. Consider a point
  $y_0$ with $\abs{y_0} = r(x_0)$, and take a sequence $y_n \to y_0$ such that
  $\abs{y_n} < r(x_0)$. Note that $\abs{q_h(x_0,y_n)} \leq \frac{\varepsilon}{2}$
  by~\eqref{eq:bound-for-z-for-large-h}. Without loss of generality, we can
  assume that $q_h(x_0,y_n)$ converges to some $q_0$ with $\abs{q_0} \leq \frac{\varepsilon}{2}$ as
  $n \to \infty$ (by compactness). By continuity, we have
  $q_0 = f_{x_0, q_0}^{(h-1)}(y_0)$. Since $(x_0,y_0,q_0) \in \Xi_\varepsilon$,
  we can still use the analytic implicit function theorem together
  with~\eqref{eq:derivative-does-not-vanish} to conclude that
  there is a neighbourhood of $(x_0,y_0,q_0)$ where the equation $f_{x, z}^{(h-1)}(y) = z$
  has exactly one solution $z$ for every $x$ and $y$, and an analytic
  function $\tilde{q}_h(x,y)$ such that
  $\tilde{q}_h(x,y) = f_{x, \tilde{q}_h(x,y)}^{(h-1)}(y)$ and
  $\tilde{q}_h(x_0,y_0) = q_0$. 
  We assume the neighbourhood to be chosen small enough such that 
  $\tilde{q}_h(x, y)\in\Xi_\varepsilon^{(3)}$ for all $(x, y)$ 
  in the neighbourhood.
  For large enough $n$, this neighbourhood
  contains $(x_0,y_n,q_h(x_0,y_n))$, so we must have
  $q_h(x_0,y_n) = \tilde{q}_h(x_0,y_n)$ for all those $n$. This implies that
  $\tilde{q}_h$ is an analytic continuation of $q_h$ in a neighbourhood of
  $(x_0, y_0)$ with values in $\Xi_\varepsilon^{(3)}$. 
  Since $y_0$ was arbitrary, we have reached the desired contradiction.

  So we conclude that there is indeed such an analytic function $q_h$ defined
  on all of $\Xi_\varepsilon^{(1,2)}$, with values in
  $\Xi_\varepsilon^{(3)}$. The fact that $q_h(x,y) = O(\lambda^h)$ finally
  follows from Lemma~\ref{lem:z-is-small}.
\end{proof}

\subsection{Location of the dominant singularity}

Let us summarise what has been proven so far. By~\eqref{eq:short-system-1} and
Lemma~\ref{lem:q_h-unique}, for sufficiently large $h$ we can express $Y_{h,h}$
in terms of $Y_{h,1}$ as
\begin{equation*}
Y_{h,h}(x) = q_h(x,Y_{h,1}(x))
\end{equation*}
at least in a neighbourhood of $0$, which we can plug into~\eqref{eq:short-system-2} to get
\begin{equation*}
Y_{h, 1}(x) = x \big(\Phi(Y_{h, 1}(x)+x)-\Phi(q_h(x,Y_{h,1}(x)))\big).
\end{equation*}
Setting
\[F_h(x,y) = x(\Phi(y+x) - \Phi(q_h(x,y))),\]
this can be rewritten as
\[ Y_{h,1}(x) = F_h(x,Y_{h,1}(x)).\]
The
function $F_h$ is analytic on
$\Xi_\varepsilon^{(1,2)}$ by Lemma~\ref{lem:q_h-unique} and the fact that 
$\Phi$ is analytic for these arguments. Note also that
\begin{equation*}
\lim_{h \to \infty} F_h(x,y) = x \big(\Phi(y+x)-1\big) \eqqcolon F_{\infty}(x,y)
\end{equation*}
pointwise for $(x, y)\in\Xi_{\varepsilon}^{(1, 2)}$.
By the estimate on $q_h$ in Lemma~\ref{lem:q_h-unique}, we also have 
\begin{equation}\label{eq:Fh_approx}
F_h(x,y) = F_{\infty}(x,y) + O(\lambda^h),
\end{equation}
uniformly for $(x,y) \in \Xi_\varepsilon^{(1,2)}$. Using the same argument as
in Lemma~\ref{lem:Lambda-small}, we can also assume (redefining $\varepsilon$
if necessary) that
\begin{equation}\label{eq:pd_approx}
\frac{\partial}{\partial y} F_h(x,y) = \frac{\partial}{\partial y} F_{\infty}(x,y) + O(\lambda^h)
\end{equation}
and analogous estimates for any finite number of partial derivatives hold as
well. Having reduced the original system of equations to a single equation for
$Y_{h,1}(x)$, we now deduce properties of its dominant singularity. Since
$Y_{h,1}(x)$ has a power series with nonnegative coefficients, by Pringsheim's
theorem it must have a dominant positive real singularity that we denote by
$\rho_h$. Since the coefficients of $Y_{h,1}(x)$ are bounded above by those of
$Y(x)$, we also know that $\rho_h \geq \rho$.

\begin{lemma}\label{lemma:singularity-location}
  For every sufficiently large $h$, $\rho_h \leq \rho +
  \lambda^{h/2}$. Moreover,
  $\eta_{h,1} := Y_{h,1}(\rho_h) = \tau - \rho + O(\lambda^{h/2})$.
\end{lemma}

\begin{proof}
  Note first that $Y_{h,1}(x)$ is
  an increasing function of $x$ for positive real $x < \rho_h$. Let
  $\tilde{\rho} = \min(\rho_h,\rho + \frac{\varepsilon}{2})$. Suppose first
  that
  $\lim_{x \to \tilde{\rho}^-} Y_{h,1}(x) \geq \tau - \rho +
  \frac{\varepsilon}{2}$. If $h$ is large enough, this implies together
  with~\eqref{eq:pd_approx} that
\begin{align*}
\lim_{x \to \tilde{\rho}^-}  \frac{\partial F_h}{\partial y} (x,Y_{h,1}(x)) &= \lim_{x \to \tilde{\rho}^-}  \frac{\partial F_{\infty}}{\partial y} (x,Y_{h,1}(x)) + O(\lambda^h) \\
&\geq \frac{\partial F_{\infty}}{\partial y} \Big(\rho,\tau -\rho + \frac{\varepsilon}{2}\Big) + O(\lambda^h) \\
&= \rho \Phi'\Big( \tau + \frac{\varepsilon}{2}\Big) + O(\lambda^h) > \rho \Phi'(\tau) = 1.
\end{align*}
On the other hand, we also have
\begin{align*}
\frac{\partial F_h}{\partial y} (\rho/2,Y_{h,1}(\rho/2)) &= \frac{\partial F_{\infty}}{\partial y} (\rho/2,Y_{h,1}(\rho/2)) + O(\lambda^h) \\
&\leq \frac{\partial F_{\infty}}{\partial y} (\rho/2,Y(\rho/2)-\rho/2) + O(\lambda^h) \\
&<  \rho \Phi'(\tau) = 1,
\end{align*}
so by continuity there must exist some $x_0 \in (\rho/2, \tilde{\rho})$ such that
\begin{equation*}
\frac{\partial F_h}{\partial y} (x_0,Y_{h,1}(x_0)) = 1.
\end{equation*}
Moreover, if $h$ is large enough we have
\begin{equation*}
\frac{\partial^2 F_h}{\partial y^2} (x_0,Y_{h,1}(x_0)) = \frac{\partial^2 F_{\infty}}{\partial y^2} (x_0,Y_{h,1}(x_0)) + O(\lambda^h) > 0
\end{equation*}
as $x_0$ and thus also $Y_{h,1}(x_0)$ are bounded below by positive constants,
and analogously $\frac{\partial F_h}{\partial x} (x_0,Y_{h,1}(x_0)) > 0$. But
this would mean that $Y_{h,1}$ has a square root singularity at $x_0 < \rho_h$
(compare the discussion in Section~\ref{sec:asy-area} later), and we reach a
contradiction. Hence we can assume that
\begin{equation}\label{eq:at_rho_tilde}
\lim_{x \to \tilde{\rho}^-} Y_{h,1}(x) < \tau - \rho + \frac{\varepsilon}{2}.
\end{equation}
Assume next that $\rho_h > \rho + \lambda^{h/2}$. Now for
$x_1 = \rho + \lambda^{h/2} < \tilde{\rho}$ (the inequality holds if $h$ is
large enough to make $\lambda^{h/2} < \frac{\varepsilon}{2}$),
$u_1 = Y_{h,1}(x_1) + x_1$ satisfies
\begin{equation}\label{eq:aux_eq_rhoh-est}
u_1 = x_1 \Phi(u_1) + O(\lambda^h),
\end{equation}
since
$F_h(x,y) = F_{\infty}(x,y) + O(\lambda^h) = x (\Phi(y+x)-1) +
O(\lambda^h)$. Note here that
$u_1 \leq \tau + \frac{\varepsilon}{2} + \lambda^{h/2}$
by~\eqref{eq:at_rho_tilde}, thus $u_1$ is in the region of analyticity of
$\Phi$ (again assuming $h$ to be large enough). However, since
$u \leq \rho \Phi(u)$ for all positive real $u$ for which $\Phi(u)$ is well-defined (the
line $u \mapsto \frac{u}{\rho}$ is a tangent to the graph of the convex
function $\Phi$ at $\tau$), for sufficiently large $h$ the right-hand side
in~\eqref{eq:aux_eq_rhoh-est} is necessarily greater than the left, and we
reach another contradiction. So it follows that
$\rho_h \leq \rho + \lambda^{h/2}$, and in particular
$\tilde{\rho} = \rho_h < \rho + \frac{\varepsilon}{2}$ if $h$ is large
enough. Since we know that
$\lim_{x \to \tilde{\rho}^-} Y_{h,1}(x) < \tau - \rho + \frac{\varepsilon}{2}$,
we also have
$\eta_{h,1} := Y_{h,1}(\rho_h) < \tau-\rho + \frac{\varepsilon}{2}$. We
conclude that $(\rho_h,\eta_{h,1}) \in \Xi_{\varepsilon}^{(1,2)}$, i.e.,
$(\rho_h,\eta_{h,1})$ lies within the region of analyticity of $F_h$. So the
singularity at $\rho_h$ must be due to the implicit function theorem failing at
this point:
\begin{equation*}
\eta_{h,1} = F_h(\rho_h,\eta_{h,1}) \text{ and } 1 = \frac{\partial F_h}{\partial y} (\rho_h,\eta_{h,1}).
\end{equation*}
The second equation in particular gives us
\begin{equation*}
\rho_h \Phi'(\eta_{h,1} + \rho_h) = 1 + O(\lambda^h)
\end{equation*}
by~\eqref{eq:pd_approx}. Since $\Phi'$ is increasing for positive real arguments and we know that $\rho \Phi'(\tau) = 1$ and $\rho_h = \rho + O(\lambda^{h/2})$, we can conclude from this that $\eta_{h,1} = \tau - \rho + O(\lambda^{h/2})$.
\end{proof}

As we have established that $\eta_{h,1} \to \tau - \rho$ as $h \to \infty$, we
will use the abbreviation $\eta_1 := \tau - \rho$ in the following. This will
later be generalised to $\eta_{h,k} := Y_{h,k}(\rho_h) \to \eta_k$, see
Sections~\ref{sec:case-exp} and~\ref{sec:case-double-exp}. For our next step,
we need a multidimensional generalisation of Rouch\'e's theorem:

\begin{theorem}[{see \cite[p.20, Theorem 2.5]{Aizenberg-Yuzhakov:1983:integral}}]\label{thm:mult_rouche}
    Let $\Omega$ be a bounded domain in $\C^n$ whose boundary $\partial \Omega$ is
  piecewise smooth. Suppose that $u,v: \overline{\Omega} \to \C^n$ are analytic
  functions, and that the boundary of $\Omega$ does not contain any zeros of
  $u$. Moreover, assume that for every $z \in \partial \Omega$, there is at least
  one coordinate $j$ for which $\abs{u_j(z)} > \abs{v_j(z)}$ holds. Then $u$
  and $u+v$ have the same number of zeros in $\Omega$.
\end{theorem}

\begin{lemma}\label{lemma:unique-singularity}
  If $\varepsilon$ is chosen sufficiently small and $h$ sufficiently large,
  then the pair $(\rho_h,\eta_{h,1})$ is the only solution to the simultaneous
  equations $F_h(x,y) = y$ and $\frac{\partial}{\partial y}F_h (x,y) = 1$ with
  $(x,y) \in \Xi_{\varepsilon}^{(1,2)}$.
\end{lemma}

\begin{proof}
  Note that $(\rho,\eta_1)$ is a solution to the simultaneous equations
  $F_{\infty}(x,y) = x(\Phi(x+y) - 1) = y$ and
  $\frac{\partial}{\partial y}F_{\infty} (x,y) = x \Phi'(x+y) = 1$, and that
  there is no other solution with $\abs{x} \leq \rho + \varepsilon$ and
  $\abs{y} \leq \eta_1 + \varepsilon$ if $\varepsilon$ is chosen sufficiently
  small by our assumptions on the function $\Phi$ (see
  Section~\ref{sec:sg_trees_basic}). We take $\Omega = \Xi_{\varepsilon}^{(1,2)}$ in
  Theorem~\ref{thm:mult_rouche} and set
\begin{equation*}
u(x,y) = \Big(F_{\infty}(x,y) - y,\frac{\partial}{\partial y}F_{\infty}(x,y) - 1 \Big).
\end{equation*}
Moreover, take 
\begin{equation*}
v(x,y) = \Big(F_{h}(x,y) - F_{\infty}(x,y), \frac{\partial}{\partial y}  F_{h}(x,y) - \frac{\partial}{\partial y} F_{\infty}(x,y) \Big).
\end{equation*}
Note that both coordinates of $v$ are $O(\lambda^h)$ by~\eqref{eq:Fh_approx}
and~\eqref{eq:pd_approx}. Since the boundary $\partial \Omega$ contains no zeros of
$u$, if we choose $h$ sufficiently large, then the conditions of
Theorem~\ref{thm:mult_rouche} are satisfied. Consequently, $u$ and $u+v$ have
the same number of zeros in $\Omega$, namely $1$. Solutions to the simultaneous
equations $F_h(x,y) = y$ and $\frac{\partial}{\partial y} F_h(x,y) = 1$ are
precisely zeros of $u+v$, so this completes the proof.
\end{proof}

At this point, it already follows from general principles (see the discussion
in~\cite[Chapter VII.4]{Flajolet-Sedgewick:ta:analy}) that for every
sufficiently large $h$, $Y_{h,1}$ has a dominant square root singularity at
$\rho_h$, and is otherwise analytic in a domain of the
form~\eqref{eq:delta-domain}. As we will need uniformity of the asymptotic
expansion and a uniform bound for the domain of analyticity, we will make this
more precise in the following section.

\subsection{Asymptotic expansion and area of analyticity}\label{sec:asy-area}

\begin{lemma}\label{lemma:factorisation-implicit-equation}
  Let $\varepsilon>0$ be such that all previous lemmata hold. There exist $\delta_1, \delta_2 > 0$, some positive number $h_0$, and
  analytic functions $R_h$ on $D_{\delta_1}(\rho_h)\times
  D_{\delta_2}(\eta_{h,1})$ and $S_h$ on $D_{\delta_2}(\eta_{h,1})$ for $h\ge
  h_0$ such that $\delta_2<\varepsilon$, $D_{\delta_1}(\rho_h)\times  D_{\delta_2}(\eta_{h,1})\subseteq
  \Xi_{\varepsilon}^{(1, 2)}$ and
  \begin{equation}\label{eq:implicit-equation-rewritten}
    F_h(x,y)-y = (x-\rho_h)R_h(x, y) + (y-\eta_{h,1})^2S_h(y)
  \end{equation}
  holds for $(x, y)\in D_{\delta_1}(\rho_h)\times D_{\delta_2}(\eta_{h,1})$ and $h\ge h_0$ and such that
  $\abs{R_h}$ is bounded from above and below by positive constants on $D_{\delta_1}(\rho_h)\times D_{\delta_2}(\eta_{h,1})$ for $h\ge h_0$ (uniformly in $h$)
  and $\abs{S_h}$ is bounded from above and below by positive constants on
  $D_{\delta_2}(\eta_{h,1})$ for $h\ge h_0$ (uniformly in $h$).

  Furthermore, the sequences $R_h$ and $S_h$ converge uniformly to some analytic functions
  $R$ and $S$, respectively. The same holds for their partial derivatives.
\end{lemma}
\begin{proof}
  Recall that we can approximate partial derivatives of $F_h$ by those of $F_{\infty}$ with an exponential error bound (as in~\eqref{eq:pd_approx}), giving us
  \begin{align*}
    \frac{\partial}{\partial x}F_h(x,y)&=\frac{\partial}{\partial x}F_\infty(x,y) + O(\lambda^h)\\&=
    \frac{\partial F_\infty}{\partial x}(\rho, \eta_1) + O(\lambda^h) + O(x-\rho) + O(y-\eta_1)\\&=\Phi(\tau) + O(\lambda^h)+O(x-\rho) + O(y-\eta_1),
  \end{align*}
  as well as
  \begin{align*}
    \frac{\partial^2}{\partial y^2}F_h(x,y)&=\frac{\partial^2}{\partial y^2}F_\infty(x,y) + O(\lambda^h)\\
                                           &=\frac{\partial^2 F_\infty}{\partial y^2}(\rho, \eta_1) + O(\lambda^h)+ O(x-\rho) + O(y-\eta_1)&\\
    &=\rho\Phi''(\tau) + O(\lambda^h)+ O(x-\rho) + O(y-\eta_1)
  \end{align*}
  for $(x, y)$ in a neighbourhood of $(\rho, \eta_1)$ contained in $\Xi_\varepsilon^{(1, 2)}$ and $h\to\infty$.

  Using Lemma~\ref{lemma:singularity-location}, we choose $\delta_1>0$ and
  $\delta_2>0$ small enough and $h_0$ large enough such that $\abs{x-\rho_h}\le
  \delta_1$, $\abs{y-\eta_{h,1}}\le \delta_2$, and $h\ge h_0$ imply that
  \begin{equation}\label{eq:F_h-x-estimate}
    \abs[\Big]{\frac{\partial}{\partial x}F_h(x, y)-\Phi(\tau)}\le \frac12\Phi(\tau)
  \end{equation}
  and
  \begin{equation}\label{eq:F_h-y-y-estimate}
    \abs[\Big]{\frac{\partial^2}{\partial y^2}F_h(x,y)-\rho\Phi''(\tau)}\le \frac12 \rho\Phi''(\tau),
  \end{equation}
  and such that $\overline{D_{\delta_1}(\rho_h) \times D_{\delta_2}(\eta_{h,1})}\subseteq
  \Xi_\varepsilon^{(1, 2)}$. By Lemma~\ref{lemma:unique-singularity}, we have
  \begin{align}
    F_h(\rho_h, \eta_{h,1})&=\eta_{h,1},\label{rho-h-eta-h-0}\\
    \frac{\partial F_h}{\partial y}(\rho_h, \eta_{h,1}) &=1.\label{rho-h-eta-h-1}
  \end{align}
  We now define
  \begin{equation*}
    S_h(y)\coloneqq \frac{F_h(\rho_h, y)-y}{(y-\eta_{h,1})^2}
  \end{equation*}
  for $y\in \overline{D_{\delta_2}(\eta_{h,1})} \setminus \{\eta_{h,1}\}$. By~\eqref{rho-h-eta-h-0} and
  \eqref{rho-h-eta-h-1}, $S_h$ has a removable singularity at
  $\eta_{h,1}$. Therefore it is analytic on $D_{\delta_2}(\eta_{h,1})$.
  By~\eqref{rho-h-eta-h-0}, we have
  \begin{align*}
    F_h(\rho_h, y)-y &= (F_h(\rho_h, y)-y) - (F_h(\rho_h, \eta_{h,1})-\eta_{h,1})\\
    &=\int_{\eta_{h,1}}^{y} \Bigl(\frac{\partial}{\partial w}F_h(\rho_h, w)-1\Bigr)\,dw.
  \end{align*}
  By~\eqref{rho-h-eta-h-1}, this can be rewritten as
  \begin{align*}
    F_h(\rho_h, y)-y
    &=\int_{\eta_{h,1}}^{y} \Bigl(\Bigl(\frac{\partial F_h}{\partial y}(\rho_h, w)-1\Bigr) - \Bigl(\frac{\partial F_h}{\partial y}(\rho_h, \eta_{h, 1})-1\Bigr)\Bigr)\,dw\\
    &=\int_{\eta_{h,1}}^{y}\int_{\eta_{h, 1}}^w \frac{\partial^2F_h}{\partial y^2}(\rho_h, v)\,dv\,dw\\
    &=\int_{\eta_{h,1}}^{y}\int_{\eta_{h, 1}}^w \rho\Phi''(\tau)\,dv\,dw+
      \int_{\eta_{h,1}}^{y}\int_{\eta_{h, 1}}^w \Bigl(\frac{\partial^2F_h}{\partial y^2}(\rho_h, v)-\rho\Phi''(\tau)\Bigr)\,dv\,dw\\
    &= \frac12\rho\Phi''(\tau)(y-\eta_{h, 1})^2+
      \int_{\eta_{h,1}}^{y}\int_{\eta_{h, 1}}^w \Bigl(\frac{\partial^2F_h}{\partial y^2}(\rho_h, v)-\rho\Phi''(\tau)\Bigr)\,dv\,dw.
  \end{align*}
  Rearranging and using the definition of $S_h(y)$ as well as~\eqref{eq:F_h-y-y-estimate} yields
  \begin{equation*}
    \abs[\Big]{S_h(y)-\frac12\rho\Phi''(\tau)}\le \frac14 \rho\Phi''(\tau)
  \end{equation*}
  for all $y\in \overline{D_{\delta_2}(\eta_{h,1})}$ and $h\ge h_0$. Thus
  $\abs{S_h(y)}$ is bounded from below and above by positive constants for every
  such $y$ and $h$.

  We now define $R_h(x, y)$ such that~\eqref{eq:implicit-equation-rewritten} holds,
  which is equivalent to
  \begin{equation*}
    R_h(x, y) \coloneqq \frac{F_h(x,y)-F_h(\rho_h, y)}{x-\rho_h}
  \end{equation*}
  for $x\in \overline{D_{\delta_1}(\rho_h)}\setminus\{\rho_h\}$ and $y\in \overline{D_{\delta_2}(\eta_{h,1})}$.
  We have
  \begin{align*}
    F_h(\rho_h, y)-F_h(x,y)&=\int_x^{\rho_h}\frac{\partial F_h}{\partial x}(w,
                             y)\,dw\\
                           &=\Phi(\tau)(\rho_h-x)+\int_x^{\rho_h}\Bigl(\frac{\partial F_h}{\partial x}(w,
                             y)-\Phi(\tau)\Bigr)\,dw.
  \end{align*}
  Rearranging and using the definition of $R_h(x, y)$ yields
  \begin{equation*}
    \abs{R_h(x, y)-\Phi(\tau)}\le \frac12 \Phi(\tau)
  \end{equation*}
  by~\eqref{eq:F_h-x-estimate}
  for  $x\in \overline{D_{\delta_1}(\rho_h)}\setminus\{\rho_h\}$ and $y\in
  \overline{D_{\delta_2}(\eta_{h,1})}$ and $h\ge h_0$. In other words, $\abs{R_h(x,y)}$ is
  bounded from below and above by positive constants for these $(x, y)$ and $h$.

  To prove analyticity of $R_h$, we use Cauchy's formula to rewrite it as
  \begin{equation*}
    R_h(x, y) = \frac{1}{2\pi i}\oint_{\abs{\zeta-\rho_h}=\delta_1} \frac{F_h(\zeta, y)-F_h(\rho_h,
      y)}{\zeta-\rho_h}\,\frac{d\zeta}{\zeta-x}
  \end{equation*}
  for $x \neq \rho_h$ (note that the integrand has a removable singularity at
  $\zeta=\rho_h$ in this case). The integral is also defined for $x=\rho_h$ and
  clearly defines an analytic function on $D_{\delta_1}(\rho_h)\times
  D_{\delta_2}(\eta_{h,1})$ whose absolute value is bounded from above and below
  by a constant.

  To see uniform convergence of $R_h$, we use Cauchy's formula once more and
  get
  \begin{equation}
    R_h(x, y) = \frac{1}{(2\pi i)^2}\oint_{\abs{\zeta-\rho_h}=\delta_1}\oint_{\abs{\eta-\eta_{h,1}}=\delta_2} \frac{F_h(\zeta, \eta)-F_h(\rho_h,
      \eta)}{\zeta-\rho_h}\,\frac{d\eta}{\eta-y}\,\frac{d\zeta}{\zeta-x} \label{eq:Rh-Cauchy}
  \end{equation}
  for $x\in D_{\delta_1}(\rho_h)$ and $y\in
  D_{\delta_2}(\eta_{h,1})$. Without loss of generality, $h_0$ is large
  enough such that $\abs{\rho_h-\rho}<\delta_1/4$ and
  $\abs{\eta_{h,1}-\eta_1}<\delta_2/4$. By Cauchy's theorem, we can change the
  contour of integration such that~\eqref{eq:Rh-Cauchy} implies
  \begin{equation*}
    R_h(x, y) = \frac{1}{(2\pi i)^2}\oint_{\abs{\zeta-\rho}=\delta_1/2}\oint_{\abs{\eta-\eta_{1}}=\delta_2/2} \frac{F_h(\zeta, \eta)-F_h(\rho_h,
      \eta)}{\zeta-\rho_h}\,\frac{d\eta}{\eta-y}\,\frac{d\zeta}{\zeta-x}
  \end{equation*}
  for $x\in D_{\delta_1/4}(\rho)$ and $y\in
  D_{\delta_2/4}(\eta_{1})$, as the deformation is happening within
  the region of analyticity of the integrand.
  Using~\eqref{eq:Fh_approx} and the fact that the
  denominator of the integrand is bounded away from zero shows that
  \begin{equation*}
    R_h(x, y) = \frac{1}{(2\pi i)^2}\oint_{\abs{\zeta-\rho}=\delta_1/2}\oint_{\abs{\eta-\eta_{1}}=\delta_2/2} \frac{F_\infty(\zeta, \eta)-F_\infty(\rho_h,
      \eta)}{\zeta-\rho_h}\,\frac{d\eta}{\eta-y}\,\frac{d\zeta}{\zeta-x} + O(\lambda^h)
  \end{equation*}
  for $x\in D_{\delta_1/4}(\rho)$ and $y\in
  D_{\delta_2/4}(\eta_{1})$. By Lemma~\ref{lemma:singularity-location},
  replacing the remaining occurrences of $\rho_h$ by $\rho$ induces another error
  term of $O(\lambda^{h/2})$, so that we get
    \begin{equation*}
    R_h(x, y) = R(x, y) + O(\lambda^{h/2})
  \end{equation*}
  with
  \begin{equation*}
    R(x, y) \coloneqq \frac{1}{(2\pi i)^2}\oint_{\abs{\zeta-\rho}=\delta_1/2}\oint_{\abs{\eta-\eta_{1}}=\delta_2/2} \frac{F_\infty(\zeta, \eta)-F_\infty(\rho,
      \eta)}{\zeta-\rho}\,\frac{d\eta}{\eta-y}\,\frac{d\zeta}{\zeta-x}
  \end{equation*}
  for $x\in D_{\delta_1/4}(\rho)$ and $y\in
  D_{\delta_2/4}(\eta_{1})$. Of course, the $O$ constants do not
  depend on $x$ and $y$; therefore, we have uniform convergence.
  Analogously, we get
  \begin{align}
    S_h(y) &=\frac{1}{2\pi i} \oint_{\abs{\eta-\eta_{h,1}}=\delta_2} \frac{F_h(\rho_h, \eta)-\eta}{(\eta-\eta_{h,1})^2}\,\frac{d\eta}{\eta-y} \label{eq:Sh-Cauchy}\\
    &=S(y) + O(\lambda^{h/2})
  \end{align}
  with
  \begin{equation*}
    S(y)\coloneqq \frac{1}{2\pi i} \oint_{\abs{\eta-\eta_{1}}=\delta_2/2} \frac{F_h(\rho, \eta)-\eta}{(\eta-\eta_{1})^2}\,\frac{d\eta}{\eta-y},
  \end{equation*}
  for $y\in
  D_{\delta_2/4}(\eta_{1})$.
  Analogous results hold for partial derivatives.

  We replace $\delta_1$ by $\delta_1/4$ and $\delta_2$ by $\delta_2/4$ to get
  the result as stated in the lemma.
\end{proof}

\begin{lemma}\label{lemma:munchhausen}
  The constants  $\delta_1$, $\delta_2$ and $h_0$ in
  Lemma~\ref{lemma:factorisation-implicit-equation} can be chosen such that
  whenever $y=F_h(x, y)$ for some $(x, y)\in D_{\delta_1}(\rho_h)\times
  D_{\delta_2}(\eta_{h,1})$ and some $h\ge h_0$, we have $\abs{y-\eta_{h,1}}<\delta_2/2$.
\end{lemma}
\begin{proof}
  We first choose $\delta_1$ and $\delta_2$ as in
  Lemma~\ref{lemma:factorisation-implicit-equation}.
  Then $y=F_h(x,y)$ and Lemma~\ref{lemma:factorisation-implicit-equation} imply
  that
  \begin{equation*}
    \abs{y-\eta_{h,1}}=\sqrt{\abs{x-\rho_h}\abs[\Big]{\frac{R_h(x, y)}{S_h(y)}}}.
  \end{equation*}

  The fraction on the right-hand side is bounded by some absolute constant
  according to Lemma~\ref{lemma:factorisation-implicit-equation}. So by
  decreasing $\delta_1$ if necessary, the right-hand side is at most $\delta_2/2$.
\end{proof}

\begin{lemma}\label{lemma:expansion-around-the-singularity}
Let $\varepsilon>0$ be such that the previous lemmata hold.
There exists $\delta_0 > 0$ such that, for all sufficiently large $h$, the asymptotic formula
\begin{equation}\label{eq:Yh1-singular-expansion}
Y_{h,1}(x) = \eta_{h,1} + a_h \Bigl(1-\frac{x}{\rho_h}\Bigr)^{1/2} + b_h\Bigl(1-\frac{x}{\rho_h}\Bigr) 
+ c_h \Bigl(1-\frac{x}{\rho_h}\Bigr)^{3/2} + O\Bigl((\rho_h-x)^2\Bigr)
\end{equation}
holds for $x\in D_{\delta_0}(\rho_h)$ with $\abs{\Arg(x-\rho_h)}\ge \pi/4$ and
certain sequences $a_h$, $b_h$ and $c_h$. The $O$-constant is independent of
$h$, and $a_h$, $b_h$, $c_h$ converge to the coefficients $a$, $b$, $c$ in
\eqref{eq:singular_expansion_Y} at an exponential rate as $h \to \infty$.
Additionally, $\abs{Y_{h,1}(x)-\eta_1}<\varepsilon/2$ for all these $x$.
\end{lemma}
\begin{proof}
By~\eqref{eq:implicit-equation-rewritten}, the function $Y_{h,1}$ is determined by the implicit equation
\begin{equation}\label{eq:Yh1-implicit-equation}
0 = F_h(x,Y_{h,1}(x)) - Y_{h,1}(x) = (x-\rho_h)R_h(x,Y_{h,1}(x)) + (Y_{h,1}(x)-\eta_{h,1})^2 S_h(Y_{h,1}(x)).
\end{equation}

  For $r>0$, set $C(r)\coloneqq \{x\in D_r(\rho_h)\colon
  \abs{\Arg(x-\rho_h)}\ge \pi/4\}$ and $\widetilde{C}(r)\coloneqq \{x\in \C\colon
  \abs{x-\rho_h}=r \text{ and }
  \abs{\Arg(x-\rho_h)}\ge \pi/4\}$.
  Choose $\delta_1$, $\delta_2$, $h_0$ as in Lemma~\ref{lemma:munchhausen}. For
  some $h\ge h_0$, let $r_h$ be the supremum of all $r\le \delta_1$ such that
  $Y_{h,1}$ can be continued analytically to $C(r)$ with values in $D_{\delta_2/2}(\eta_{h,1})$. We
  claim that $r_h=\delta_1$.

  Suppose for contradiction that $r_h<\delta_1$ and let
  $x_{\infty}\in\widetilde C(r_h)$. Choose a sequence of elements
  $x_n\in C(r_h)$ converging to
  $x_\infty$ for $n\to \infty$ and set $y_n\coloneqq Y_{h, 1}(x_n)$ for all $n$. By assumption, we have
  $\abs{y_n-\eta_{h,1}}\le \delta_2/2$. By replacing the sequence $x_n$ by a subsequence
  if necessary, we may assume that the sequence $y_n$ is convergent to some
  limit $y_{\infty}$. Note that $\abs{y_\infty-\eta_{h,1}}\le
  \delta_2/2$. By continuity of $F_h$, we also have $y_\infty=F_h(x_\infty,
  y_\infty)$. As $(x_\infty, y_\infty)\in \Xi_\varepsilon^{(1, 2)}$ with
  $x_\infty\neq \rho_h$, Lemma~\ref{lemma:unique-singularity} and the analytic
  implicit function theorem imply that $Y_{h,1}$ can be continued analytically
  in a suitable open neighbourhood of $x_\infty$. This neighbourhood can be chosen
  small enough such that the inequality $\abs{Y_{h,1}(x)-\eta_{h,1}}\le \delta_2$ holds for all $x$
  in this neighbourhood. However, Lemma~\ref{lemma:munchhausen} implies that we
  then actually have $\abs{Y_{h,1}(x)-\eta_{h,1}}\le \delta_2/2$ for all such $x$.

  The set of these open neighbourhoods associated with all $x_\infty\in
  \widetilde{C}(r_h)$ covers the compact set $\widetilde{C}(r_h)$, so a finite
  subset of these open neighbourhoods can be selected. Thus we find an analytic
  continuation of $Y_{h,1}$ to $C(\widetilde{r}_h)$ for some
  $\widetilde{r}_h\in (r_h, \delta_1)$
  with values still in $D_{\delta_2/2}(\eta_{h,1})$, which is a
  contradiction to the choice of $r_h$.

  Thus we have $r_h=\delta_1$. In particular, choosing $h$ large enough that $\abs{\eta_{h, 1} - \eta_1} < (\varepsilon - \delta_2)/2$ gives
  $\abs{Y_{h,1}(x)-\eta_1} \le \abs{Y_{h,1}(x) - \eta_{h,1}} + \abs{\eta_{h,1}-\eta_1} < \delta_2/2 + (\varepsilon-\delta_2)/2 = \varepsilon/2$ for all $x\in C(\delta_1)$.

  Rearranging~\eqref{eq:Yh1-implicit-equation} yields
  \begin{equation}\label{eq:expansion-around-singularity-bootstrap-preparation}
    (\eta_{h, 1} - Y_{h,1}(x))^2 = \Bigl(\rho_h\frac{R_h(x, Y_{h,1}(x))}{S_h(Y_{h,
        1}(x))}\Bigr) \Bigl(1-\frac{x}{\rho_h}\Bigr).
  \end{equation}
  We know from Lemma~\ref{lemma:factorisation-implicit-equation} that $R_h$ is
  bounded above and $S_h$ is bounded below on
  $D_{\delta_1}(\rho_h)\times D_{\delta_2}(\eta_{h,1})$ and
  $D_{\delta_2}(\eta_{h,1})$, respectively.  Therefore, the absolute value of
  the first factor on the right-hand side
  of~\eqref{eq:expansion-around-singularity-bootstrap-preparation} is bounded
  above and below by positive constants for $x\in D_{\delta_1}(\rho_h)$.  For
  $x<\rho_h$, we have that the factor $(1-x/\rho_h)$ is trivially positive and
  that $\eta_{h,1}>Y_{h,1}(x)$ because $Y_{h,1}$ is strictly increasing on
  $(0, \rho_h)$, so the first factor on the right-hand side
  of~\eqref{eq:expansion-around-singularity-bootstrap-preparation} must be
  positive. Thus we may take the principal value of the square root to
  rewrite~\eqref{eq:expansion-around-singularity-bootstrap-preparation} as
  \begin{equation}\label{eq:expansion-around-singularity-bootstrap}
    \eta_{h, 1} - Y_{h,1}(x) = \sqrt{\rho_h\frac{R_h(x, Y_{h,1}(x))}{S_h(Y_{h,
          1}(x))}} \Bigl(1-\frac{x}{\rho_h}\Bigr)^{1/2}
  \end{equation}
  for $x\in C(\delta_1)$. The above considerations also show that the radicand
  in \eqref{eq:expansion-around-singularity-bootstrap} remains positive in
  the limit $x\to \rho_h^{-}$ (i.e., as
  $x$ approaches $\rho_h$ from the left) and then for $h\to\infty$.

  As we just observed that the first factor on the right-hand side
  of~\eqref{eq:expansion-around-singularity-bootstrap} is bounded,
  \eqref{eq:expansion-around-singularity-bootstrap} implies
  \begin{equation}\label{eq:expansion-around-singularity-bootstrap-step-1}
    Y_{h,1}(x)-\eta_{h,1}=O\bigl((x-\rho_h)^{1/2}\bigr),
  \end{equation}
  with an $O$-constant that is independent of $h$. We can now iterate this
  argument: using Taylor expansion along with the fact that partial derivatives
  of $R_h$ and $S_h$ are uniformly bounded above while $S_h$ is also uniformly
  bounded below, we obtain
\begin{align*}
\frac{R_h(x, Y_{h,1}(x))}{S_h(Y_{h,
          1}(x))} &= \frac{R_h(\rho_h,\eta_{h,1}) + O(x-\rho_h) + O(Y_{h,1}(x)-\eta_{h,1})}{S_h(\eta_{h,1}) + O(Y_{h,1}(x)-\eta_{h,1})} \\&= \frac{R_h(\rho_h,\eta_{h,1})}{S_h(\eta_{h,1})} + O\bigl((x-\rho_h)^{1/2}\bigr).
\end{align*}
Plugging this into~\eqref{eq:expansion-around-singularity-bootstrap} yields
  \begin{equation*}
    \eta_{h, 1} - Y_{h,1}(x) = \sqrt{\rho_h\frac{R_h(\rho_h, \eta_{h,1})}{S_h(\eta_{h,1})}} \Bigl(1-\frac{x}{\rho_h}\Bigr)^{1/2} + O( x - \rho_h),
  \end{equation*}
  still with an $O$-constant that is independent of $h$. This can be continued
  arbitrarily often to obtain further terms of the expansion and an improved
  error term (for our purposes, it is enough to stop at $O((x -
  \rho_h)^2)$). Indeed it is well known (cf.~\cite[Lemma
  VII.3]{Flajolet-Sedgewick:ta:analy}) that an implicit equation of the
  form~\eqref{eq:Yh1-implicit-equation} has a solution as a power series in
  $(1-x/\rho_h)^{1/2}$. In particular,~\eqref{eq:Yh1-singular-expansion}
  follows with an error term that is uniform in $h$. The coefficients
  $a_h,b_h,c_h$ can be expressed in terms of $R_h$, $S_h$ and their partial
  derivatives evaluated at $(\rho_h,\eta_{h,1})$: specifically,
\begin{align*}
a_h &= - \sqrt{\rho_h\frac{R_h(\rho_h, \eta_{h,1})}{S_h(\eta_{h,1})}},\\
b_h &= \frac{\rho_h S_h(\eta_{h,1}) \frac{\partial R_h}{\partial y}(\rho_h,\eta_{h,1})-\rho_h S_h'(\eta_{h,1}) R_h(\rho_h,\eta_{h,1})}{2 S_h(\eta_{h,1})^2},\\
c_h &=\frac{\rho_h^{3/2}N}{8\sqrt{R_h(\rho_h, \eta_{h,1})S_h(\eta_{h,1})^7}},
\end{align*}
where the numerator $N$ is a polynomial in $R_h(\rho_h,\eta_{h,1})$,
$S_h(\eta_{h,1})$ and their derivatives.
By Lemma~\ref{lemma:factorisation-implicit-equation}, $R_h$ and $S_h$ as well
as their partial derivatives converge uniformly to $R$ and $S$ as well as their
partial derivatives, respectively, with an error bound of $O(\lambda^{h/2})$.
We also know that $\rho_h$ and
$\eta_{h,1}$ converge exponentially to $\rho$ and $\eta_1$, respectively, see
Lemma~\ref{lemma:singularity-location}. This means that first replacing
all occurrences of $R_h$ and $S_h$ by $R$ and $S$, respectively, and then replacing all
occurrences of $\rho_h$ and $\eta_{h,1}$ by $\rho$ and $\eta_1$, respectively,
shows that $a_h=a+O(\lambda^{h/2})$, $b_h=b+O(\lambda^{h/2})$, and
$c_h=c+O(\lambda^{h/2})$ where $a$, $b$, and $c$ are the results of these
replacements.
Taking the limit for $h\to\infty$
in~\eqref{eq:implicit-equation-rewritten} shows that $R$ and $S$ and therefore
$a$, $b$, and $c$ play the same role with respect to $F_\infty$ as $R_h$,
$S_h$, $a_h$, $b_h$, and $c_h$ play with respect to $F_h$, which implies that
$a$, $b$, and $c$ are indeed the constants from~\eqref{eq:singular_expansion_Y}.
\end{proof}

Having dealt with the behaviour around the singularity, it remains to prove a uniform bound on $Y_{h,1}$ in a domain of the form~\eqref{eq:delta-domain} for fixed $\delta$.

\begin{lemma}\label{lemma:continuation-to-pacman}
  Let $\varepsilon>0$ be such that all previous lemmata hold.
  There exist $\delta > 0$ and a positive integer $h_0$ such that $Y_{h,1}(x)$ has an analytic continuation to the domain
\begin{equation*}
    \{x \in \C: \abs{x} \leq (1 + \delta)\abs{\rho_h}, \abs{\Arg(x/\rho_h - 1)} > \pi/4\}
\end{equation*}
for all $h \geq h_0$, and has the uniform upper bound
\begin{equation*}
\abs{Y_{h,1}(x)} \leq \tau - \rho + \frac{\varepsilon}{2}  = \eta_1 + \frac{\varepsilon}{2}
\end{equation*}
for all $h \geq h_0$ and all $x$.
\end{lemma}

\begin{proof}
Let us define $r_h = \sup \mathcal{R}_h$, where
\begin{equation*}
\mathcal{R}_h =  \Big\{ r \,: \, Y_{h,1} \text{ extends analytically to } D_r(0) \setminus D_{\delta_0}(\rho_h) \text{ and satisfies } \abs{Y_{h,1}(x)} < \eta_1 + \frac{\varepsilon}{2} \text{ there} \Big\},
\end{equation*}
with $\delta_0$ as in the previous lemma. Note that trivially, $r_h \geq
\rho$. If $\liminf_{h \to \infty} r_h > \rho$, we are done: in this case, there
is some $\delta > 0$ such that $Y_{h,1}$ extends analytically to
$D_{\rho(1+\delta)}(0) \setminus D_{\delta_0}(\rho_h)$ and satisfies
$\abs{Y_{h,1}(x)} < \eta_1 + \frac{\varepsilon}{2}$ there. As the previous
lemma covers $D_{\delta_0}(\rho_h)$, this already completes the proof.

So let us assume that $\liminf_{h \to \infty} r_h = \rho$ and derive a
contradiction.
The assumption implies that there is an increasing sequence of
positive integers $h_j$ such that $\lim_{j\to\infty }r_{h_j}=\rho$.
Without loss of generality, we may assume that
$r_{h_j} \leq \rho + \frac{\varepsilon}{2}$ for all $j$.
Pick (for each sufficiently large
$j$) a point $x_{h_j}$ with $\abs{x_{h_j}} = r_{h_j}$ and
$\abs{Y_{{h_j},1}(x_{h_j})} = \eta_1 + \frac{\varepsilon}{2}$. If this were not
possible, we could analytically continue $Y_{{h_j},1}$ at every point $x$ with
$\abs{x} = r_{h_j}$ and $x \notin D_{\delta_0}(\rho_{h_j})$ to a disk where $Y_{h_j,1}$
is still bounded by $\eta_1 + \frac{\varepsilon}{2}$. This analytic
continuation is possible, since by Lemma~\ref{lemma:unique-singularity} the
pair $(\rho_{h_j},\eta_{{h_j},1})$ is the only solution to the simultaneous equations
$F_{h_j}(x,y) = y$ and $\frac{\partial }{\partial y} F_{h_j}(x,y) = 1$ with
$(x,y) \in \Xi_{\varepsilon}^{(1,2)}$, so the analytic implicit function
theorem becomes applicable (compare e.g.~the analytic continuation of $q_h$ in
Lemma~\ref{lem:q_h-unique}). By compactness, this would allow us to extend
$Y_{{h_j},1}$ to $D_r(0) \setminus D_{\delta_0}(\rho_{h_j})$ for some $r > r_{h_j}$ while
still maintaining the inequality
$\abs{Y_{{h_j},1}(x)} < \eta_1 + \frac{\varepsilon}{2}$, contradicting the choice
of $r_{h_j}$.

Without loss of generality (choosing a subsequence if
necessary), we can assume that $x_{h_j}$ and $Y_{h_j,1}(x_{h_j})$ have limits
$x_{\infty}$ and $y_{\infty}$, respectively. By construction,
$\abs{x_{\infty}} = \rho$ and
$\abs{y_{\infty}} = \eta_1 + \frac{\varepsilon}{2}$.

Since $x_{h_j} \notin D_{\delta_0}(\rho)$ for all $j$, $\Arg x_{h_j}$ is
bounded away from $0$. Thus we can find $\alpha > 0$ such that
$\abs{\Arg x_{h_j}} \geq 2\alpha$ for all $j$. Define the region $A$ by
\begin{equation*}
  A = \Bigl\{z \in \C\,:\, \abs{z} < \frac12 \text{ or } (\abs{z} < 1 \text{ and } \abs{\Arg z} < \alpha) \Bigr\}.
\end{equation*}

\begin{figure}
\centering
\begin{tikzpicture}

\fill[fill=gray!20] (0,0) circle (1);
\fill[fill=gray!20] (0,0) -- ({sqrt(3)},1) arc [start angle=30, delta angle=30, radius=2] -- (0,0);
\draw[thick,<->] ({3*sqrt(3)/4},3/4) arc [start angle=30, delta angle=30, radius=3/2];

\draw [->,thick] (-2.2,0)--(2.2,0);
\draw [->,thick] (0,-2.2)--(0,2.2);

\node[circle,fill,inner sep = 1] at (1.8, 0) {};
\node[circle,fill,inner sep = 1] at ({sqrt(2)}, {sqrt(2)}) {};

\node at (1.6,1.6) {$x_{h_j}$};
\node at (1.8,-0.3) {$\rho_{h_j}$};
\node at (0.9,0.9) {$2\alpha$};
        
\end{tikzpicture}
    \caption{Illustration of the domain $x_{h_j}A$.}
    \label{fig:domain_A}
\end{figure}
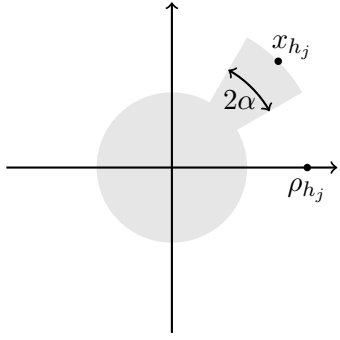

Note that $x_{h_j}A$ avoids the part of the real axis that includes
$\rho_{h_j}$ (see Figure~\ref{fig:domain_A}), so the function $Y_{h_j,1}(x)$ is
analytic in this region for all $j$ by construction since
$(x, Y_{h_j,1}(x))\in\Xi_\varepsilon^{(1,2)}$ whenever $x \in x_{h_j}A$. So we
have a sequence of functions $W_j(z) := Y_{h_j,1}(x_{h_j}z)$ that are all
analytic on $A$ and are uniformly bounded above by
$\eta_1 + \frac{\varepsilon}{2}$ by our choice of $x_{h_j}$. By Montel's
theorem, there is a subsequence of these functions (without loss of generality
the sequence itself) that converges locally uniformly and thus to an analytic
function $W_\infty$ on $A$. This function needs to satisfy the following:
\begin{itemize}
\item $W_{\infty}(0) = 0$, since $W_j(0) = 0$ for all $j$,
\item $W_{\infty}(z) = F_{\infty}(x_{\infty}z, W_{\infty}(z)) = x_{\infty}z(\Phi(x_{\infty}z + W_{\infty}(z))-1)$ for $z\in A$, since we have the uniform estimate
  \begin{equation*}W_j(z)=Y_{h_j,1}(x_{h_j}z)=F_{h_j}(x_{h_j}z, Y_{h_j,1}(x_{h_j}z))=F_{\infty}(x_{h_j}z, Y_{h_j,1}(x_{h_j}z)) + O(\lambda^h).\end{equation*}
This is also equivalent to
\begin{equation*}
x_{\infty}z + W_{\infty}(z) = x_{\infty}z\Phi(x_{\infty}z + W_{\infty}(z)).
\end{equation*}
\end{itemize}
These two properties imply that $W_{\infty}(z) =Y(x_{\infty}z) - x_{\infty}z$,
since $Y$ is the unique function that is analytic at $0$ and satisfies the
implicit equation $Y(x) = x \Phi(Y(x))$. Implicit differentiation of
$Y_{h_j,1}(x)=F_{h_j}(x, Y_{h_j, 1}(x))$ for $x \in x_{h_j}A$ yields
\begin{equation}\label{eq:Yhj-derivative}
  Y_{h_j,1}'(x) = \frac{\frac{\partial F_{h_j}}{\partial x} (x, Y_{h_j, 1}(x))}{1 - \frac{\partial F_{h_j}}{\partial y} (x, Y_{h_j, 1}(x))} = \frac{\frac{\partial F_{\infty}}{\partial x} (x, Y_{h_j, 1}(x)) + O(\lambda^{h_j})}{1 - \frac{\partial F_{\infty}}{\partial y} (x, Y_{h_j, 1}(x)) + O(\lambda^{h_j})}.
\end{equation}
Note that the numerator is uniformly bounded. Moreover, we recall again that
the only solution to the simultaneous equations
$F_{\infty}(x,y) = x(\Phi(y+x) - 1) = y$ and
$\frac{\partial }{\partial y}F_{\infty}(x,y) = x \Phi'(y+x) = 1$ with
$\abs{x} \leq \rho + \varepsilon$ and $\abs{x+y} \leq \tau + \varepsilon$ is
$(x,y) = (\rho,\tau-\rho) = (\rho,\eta_1)$ by our assumptions on $\Phi$. By
construction, there is a constant $\varepsilon_A > 0$ such that
$\abs{x - \rho} \geq \varepsilon_A$ whenever $x \in x_{h_j}A$ for some $j$.
The map $(x, y)\mapsto \| (F_{\infty}(x,y) - y, \frac{\partial }{\partial y}F_{\infty}(x,y) - 1)\|$
is continuous on the compact set
\[\mathcal{K}\coloneqq \{(x,y): \abs{x} \leq \rho + \varepsilon
  \text{ and }
  \abs{x+y} \leq \tau + \varepsilon
  \text{ and }
  \abs{x-\rho} \ge \varepsilon_A \}\]
and has no zero there (using the
Euclidean norm on $\C^2$). Therefore, it attains a minimum $\delta_A>0$ on
$\mathcal{K}$.

Now for $x \in x_{h_j}A$, $\abs{x} \leq \rho + \varepsilon$ holds by
assumption, as does $\abs{x + Y_{h_j, 1}(x)} \leq \tau +
\varepsilon$. Moreover,
$\abs{x - \rho} \geq
\varepsilon_A$. Thus we can conclude that $(x, Y_{h_j,1}(x)) \in \mathcal{K}$ and therefore
$\| (F_{\infty}(x,Y_{h_j, 1}(x)) - Y_{h_j, 1}(x), \frac{\partial
  F_{\infty}}{\partial y}(x,Y_{h_j, 1}(x)) - 1)\| \geq \delta_A$ for all such
$x$. Since
\begin{equation*}
F_{\infty}(x,Y_{h_j, 1}(x)) - Y_{h_j, 1}(x) = F_{h_j}(x,Y_{h_j, 1}(x)) - Y_{h_j, 1}(x) + O(\lambda^{h_j}) = O(\lambda^{h_j}),
\end{equation*}
this means that
$\abs{1 - \frac{\partial F_{\infty}}{\partial y} (x, Y_{h_j, 1}(x))} \geq
\delta_A - O(\lambda^{h_j})$, so that the denominator
in~\eqref{eq:Yhj-derivative} is bounded below by a positive constant for
sufficiently large $j$.

So we can conclude that $Y_{h_j,1}'(x)$ is uniformly bounded by a constant for
$x \in x_{h_j}A$, implying that $W_j'(z)$ is uniformly bounded (for all
$z \in A$ and all sufficiently large $j$) by a constant that is independent of
$j$. Therefore, $W_j(z)$ is a uniformly equicontinuous sequence of functions on
$\overline{A}$, the closure of $A$. By the Arzelà--Ascoli theorem, this implies
that $W_j(z) \to W_{\infty}(z) = Y(x_{\infty}z) - x_{\infty}z$ holds even for
all $z \in \overline{A}$, not only on $A$. In particular,
$y_{\infty} = W_{\infty}(1) = Y(x_{\infty}) - x_{\infty}$. Here, we have
$\abs{x_{\infty}} \leq \rho$ and
$\abs{y_{\infty}} = \eta_1 + \frac{\varepsilon}{2}$ by assumption. However,
\begin{equation*}
\abs{Y(x) - x} \leq \abs{Y(\rho) - \rho}  =\eta_1
\end{equation*}
holds for all $\abs{x} \leq \rho$ by the triangle inequality, so we finally reach a contradiction.
\end{proof}

We conclude this section with a summary of the results proven so far. The following proposition follows by combining the last two lemmata.

\begin{proposition}\label{prop:analytic_properties_of_Yh}
There exists a constant $\delta > 0$ such that $Y_{h,1}(x)$ can be continued analytically to the domain
\begin{equation*}
    \{x \in \C: \abs{x} \leq (1 + \delta)\abs{\rho_h}, \abs{\Arg(x/\rho_h - 1)} > \pi/4\}
\end{equation*}
for every sufficiently large $h$. Moreover, $Y_{h,1}(x)$ is then uniformly
bounded on this domain by a constant that is independent of $h$, and the
following singular expansion holds near the singularity:
  \begin{equation*}
Y_{h,1}(x) = \eta_{h,1} + a_h \Bigl(1-\frac{x}{\rho_h}\Bigr)^{1/2} + b_h\Bigl(1-\frac{x}{\rho_h}\Bigr) 
+ c_h \Bigl(1-\frac{x}{\rho_h}\Bigr)^{3/2} + O\Bigl((\rho_h-x)^2\Bigr),
  \end{equation*}
where the $O$-constant is independent of $h$ and $a_h,b_h,c_h$ converge at an exponential rate to $a,b,c$ respectively as $h \to \infty$.
\end{proposition}

\begin{remark}\label{remark:periodic-case-1}
  Let $D=\gcd\{i\in\N\colon w_i\neq 0\}$ be the \emph{period} of $\Phi$. The
  purpose of this remark is to give indications how the results so far have to be
  adapted for the case $D>1$.

  If $D>1$, then for all trees of our simply generated family of trees, the
  number $n$ of vertices will be congruent to $1$ modulo $D$ because all
  outdegrees are multiples of $D$. Trivially, the same is true for all trees
  with maximum protection number $h$.

  By \cite[Remark~VI.17]{Flajolet-Sedgewick:ta:analy}, both $Y$ and $Y_{h,1}$
  have $D$ conjugate roots on its circle of convergence. Therefore, it is
  enough to study the positive root at the radius of convergence. Up to
  Theorem~\ref{thm:mult_rouche}, no changes are required. In
  Lemma~\ref{lemma:unique-singularity}, there are exactly $D$ solutions instead
  of exactly one solution to the simultaneous
  equations. Lemmata~\ref{lemma:factorisation-implicit-equation},
  \ref{lemma:munchhausen}, and \ref{lemma:expansion-around-the-singularity}
  analyse the behaviour of $Y_{h,1}$ around the dominant positive singularity
  and remain valid without any change. In the proof of
  Lemma~\ref{lemma:continuation-to-pacman}, we need to exclude balls around the
  conjugate roots. Proposition~\ref{prop:analytic_properties_of_Yh} must also
  be changed to exclude the conjugate roots.
\end{remark}

\section{The exponential case: \texorpdfstring{$w_1 \neq 0$}{w1 not 0}}\label{sec:case-exp}

\subsection{Asymptotics of the singularities}\label{sec:singularity-asymptotics-exp}

Proposition~\ref{prop:analytic_properties_of_Yh} that concluded the previous
section shows that condition~\eqref{bullet:thm-1-2} of
Theorem~\ref{thm:pro-wags-1} is satisfied (with $\alpha = \frac12$) by the
generating functions $Y_{h,1}$ (and thus also $Y_{h,0}$, since
$Y_{h,0}(x) = Y_{h,1}(x) + x$).  It remains to study the behaviour of the
singularity $\rho_h$ of $Y_{h,0}$ and $Y_{h,1}$ to make the theorem
applicable. As it turns out, condition~\eqref{bullet:thm-1-1} of
Theorem~\ref{thm:pro-wags-1} holds precisely if vertices of outdegree $1$ are
allowed in our simply generated family of trees. In terms of the weight
generating function $\Phi$, this can be expressed as $w_1 = \Phi'(0) \neq
0$. Starting with Lemma~\ref{lem:eta-asymp}, we will assume that this holds. The case where
vertices of outdegree $1$ cannot occur (equivalently, $w_1 = \Phi'(0) = 0$) is
covered in Section~\ref{sec:case-double-exp}.

Let us define the auxiliary quantities $\eta_{h,k} := Y_{h,k}(\rho_h)$ for all
$0 \leq k \leq h$. We know that these must exist and be finite for all
sufficiently large $h$. Since the coefficients of $Y_{h,k}$ are nonincreasing
in $k$ in view of the combinatorial interpretation, we must have
\begin{equation}\label{eq:eta-monotonicity}
\eta_{h,0} \geq \eta_{h,1} \geq \cdots \geq \eta_{h,h}.
\end{equation}
Note also that the following system of equations holds:
\begin{align}
  \eta_{h, 0} & = \eta_{h, 1} + \rho_h,\label{eq:syst1}\\
  \eta_{h, k} & = \rho_h\Phi(\eta_{h, k-1}) - \rho_h\Phi(\eta_{h, h}) \qquad \text{for } 1 \leq k \leq h,\label{eq:syst2}
\end{align}
in view of~\eqref{eq:func-eq-1} and~\eqref{eq:func-eq-2}, respectively. Since
$Y_{h,1}$ is singular at $\rho_h$ by assumption, the Jacobian determinant of
the system that determines $Y_{h,0},Y_{h,1},\ldots,Y_{h,h}$ needs to vanish (as
there would otherwise be an analytic continuation by the analytic implicit
function theorem). This determinant is given by
\begin{equation*}
\begin{vmatrix}
 1 & -1 & 0 & \cdots & 0 & 0 \\ 
-\rho_h \Phi'(\eta_{h,0}) & 1 & 0 & \cdots & 0 & \rho_h \Phi'(\eta_{h,h}) \\
0 & -\rho_h \Phi'(\eta_{h,1}) & 1 & \cdots & 0 & \rho_h \Phi'(\eta_{h,h}) \\
\vdots & \vdots & \vdots & \ddots & \vdots & \vdots \\
0 & 0 & 0 & \cdots & -\rho_h \Phi'(\eta_{h,h-1}) & 1 + \rho_h \Phi'(\eta_{h,h})
\end{vmatrix}\,.
\end{equation*}
Using column expansion with respect to the last column to obtain the determinant, we find that this simplifies to 
\begin{equation}\label{eq:Jacdet}
\prod_{j=1}^{h} \big( \rho_h \Phi'(\eta_{h,j}) \big) + \big(1 - \rho_h \Phi'(\eta_{h,0}) \big) \Big( 1 + \sum_{k=2}^{h} \prod_{j=k}^{h} \big( \rho_h \Phi'(\eta_{h,j}) \big) \Big) = 0.
\end{equation}

We will now use~\eqref{eq:syst1},~\eqref{eq:syst2}, and~\eqref{eq:Jacdet} to
determine an asymptotic formula for $\rho_h$. Throughout this section, $B_i$'s
will always be positive constants with $B_i < 1$ that depend on the specific
family of simply generated trees, but nothing else.

\begin{lemma}\label{lem:exp-bound}
    There exist positive constants $C$ and $B_1$ with $B_1 < 1$ such that $\eta_{h,k} \leq C B_1^k$ for all sufficiently large $h$ and all $k$ with $0 \leq k \leq h$.
\end{lemma}
    
\begin{proof}
  Since we already know that $\eta_{h,1}$ converges to $\tau - \rho$ and that $\rho_h$ converges to $\rho$,
  $\eta_{h,0}$ converges to $\tau$ by~\eqref{eq:syst1}. By the monotonicity
  property~\eqref{eq:eta-monotonicity}, all $\eta_{h,k}$ must therefore be
  bounded by a single constant $M$ for sufficiently large $h$. Since
  $\eta_{h,1}$ converges to $\tau - \rho$, we must have that $\rho_h\Phi'(\eta_{h, 1})$ converges 
  to $\rho\Phi'(\tau-\rho)$. Therefore,  $\rho_h\Phi'(\eta_{h, 1}) \le \rho\Phi'(\tau - \rho/2)$ 
  for sufficiently large $h$. It follows that
  $\rho_h \Phi'(\eta_{h,1}) \leq \rho \Phi'(\tau - \rho/2) < \rho \Phi'(\tau) =
  1$. For all $1 \leq j \leq h$, we now have
\begin{equation*}
\eta_{h,j} = \rho_h \Phi(\eta_{h,j-1}) - \rho_h \Phi(\eta_{h,h}) \leq \rho_h \Phi'(\eta_{h,j-1}) (\eta_{h,j-1} - \eta_{h,h}) \leq \rho_h \Phi'(\eta_{h,1}) \eta_{h,j-1}.
\end{equation*}
    Thus by induction
\begin{equation*}
\eta_{h,k} \leq \eta_{h,1} \big( \rho_h \Phi'(\eta_{h,1}) \big)^{k-1} \leq M \big( \rho \Phi'(\tau - \rho/2) \big)^{k-1}.
\end{equation*}
 This proves the desired inequality for sufficiently large $h$ and $1 \leq k \leq h$ with $B_1 = \rho \Phi'(\tau - \rho/2) < 1$, and we are done.
\end{proof}

With this bound, we will be able to refine the estimates for the system of
equations, leading to better estimates for $\rho_h$ and $\eta_{h, 0}$.  Recall
from Lemma~\ref{lemma:singularity-location} that $\rho_h$ and $\eta_{h,1}$
converge to their respective limits $\rho$ and $\tau - \rho = \eta_1$ (at
least) exponentially fast. Since $\eta_{h,0} = \eta_{h,1} + \rho_h$ by
\eqref{eq:syst1}, this also applies to $\eta_{h,0}$. We show that an analogous
statement also holds for $\eta_{h,k}$ with arbitrary $k$. In view
of~\eqref{eq:syst2}, it is natural to expect that $\eta_{h,k} \to \eta_k$,
where $\eta_k$ is defined recursively as follows: $\eta_0 = \tau$
and, for $k > 0$,
$\eta_k = \rho \Phi(\eta_{k-1}) - \rho$, which also 
coincides with our earlier definition of $\eta_1 = \tau - \rho$.
This is proven in the following lemma.

\begin{lemma}\label{lem:exp-convergence}
For a suitable constant $B_2 < 1$ and sufficiently large $h$, we have $\rho_h = \rho + O(B_2^h)$ and $\eta_{h,k} = \eta_k + O(B_2^h)$ for all $k$ with $0 \leq k \leq h$, uniformly in $k$.
\end{lemma}
\begin{proof}
  For a suitable choice of $B_2$, the estimate for $\rho_h$ has been established by
  Lemma~\ref{lemma:singularity-location}, as
  has the estimate for $\eta_{h,k}$ in the cases where $k = 0$ and $k = 1$. Set
  $\delta_{h, k} = \eta_{h, k} - \eta_k$.
Since $\eta_{h,h} \leq C B_1^h$ by Lemma~\ref{lem:exp-bound}, we have $\Phi(\eta_{h,h}) = \Phi(0) + O(B_1^h) = 1 + O(B_1^h)$. Without loss of generality, suppose that $B_2 \geq B_1$. Then, using~\eqref{eq:syst2}, we obtain
  \begin{align*}
    \eta_{h, k} & = \rho_h\Phi(\eta_{h, k-1}) - \rho_h\Phi(\eta_{h, h}) \\
    & = (\rho + O(B_2^h))\Phi(\eta_{k-1}+ \delta_{h, k-1}) - (\rho + O(B_2^h))(1 + O(B_1^h))\\
    & = \rho(\Phi(\eta_{k-1}) + \Phi'(\xi_{h, k-1})\delta_{h, k-1}) - \rho + O(B_2^h) \\
    & = \eta_k + \rho \Phi'(\xi_{h,k-1}) \delta_{h,k-1} + O(B_2^h)
  \end{align*}
  where $\xi_{h, k-1}$ is between $\eta_{k-1}$ and $\eta_{h, k-1}$ (by the mean
  value theorem) and the $O$-constant is independent of $k$. Let $M$ be this
  $O$-constant. We already know (compare the proof of
  Lemma~\ref{lem:exp-bound}) that
  $\eta_{h,k-1} \leq \eta_{h,1} \leq \tau - \rho/2$ for every $k \geq 2$ if $h$
  is sufficiently large. Likewise, it is easy to see that $\eta_k$ is
  decreasing in $k$, hence $\eta_{k-1} \leq \eta_1 = \tau - \rho$. Thus,
  $\xi_{h,k-1} \leq \tau - \rho/2$ and
  $\rho\Phi'(\xi_{h, k-1}) \leq \rho\Phi'(\tau - \rho/2) = B_1 < 1$. So we
  have, for every $k > 1$,
\begin{equation*}
\abs{\delta_{h,k}} = \abs{\eta_{h,k} - \eta_k} \leq B_1 \abs{\delta_{h,k-1}} + MB_2^h.
\end{equation*}
Iterating this inequality yields
\begin{equation*}
\abs{\delta_{h,k}} \leq B_1^{k-1} \abs{\delta_{h,1}} + (1 + B_1 + \cdots + B_1^{k-2}) MB_2^h \leq \abs{\delta_{h,1}} + \frac{MB_2^h}{1-B_1},
\end{equation*}
and the desired statement follows.
\end{proof}

From Lemma~\ref{lem:exp-bound} and the fact that $\eta_{h,k} \to \eta_k$, we
trivially obtain $\eta_k \leq C \cdot B_1^k$, with the same constants $B_1$ and
$C$ as in Lemma~\ref{lem:exp-bound}. In fact, we can be more precise, and this 
is demonstrated in the lemma that follows. Since the
expression $\rho \Phi'(0)$ occurs frequently in the following, we set
$\zeta := \rho \Phi'(0)$. Recall that we assume $\Phi'(0) \neq 0$ until the end of this section.

\begin{lemma}\label{lem:eta-asymp}
The limit $\lambda_1 := \lim_{k \to \infty} \zeta^{-k} \eta_k$ exists. Moreover, we have 
\begin{equation*}
\eta_k = \lambda_1 \zeta^k (1 + O(B_1^k)),
\end{equation*}
with $B_1$ as in Lemma~\ref{lem:exp-bound}.
\end{lemma}

\begin{proof}
Recall that we defined the sequence $(\eta_k)_{k \geq 0}$ by $\eta_0 = \tau$ and $\eta_k = \rho \Phi(\eta_{k-1}) - \rho$ for $k\ge 1$.
Using Taylor expansion, we obtain
\begin{equation*}
\eta_k = \rho \Phi'(0) \eta_{k-1} (1 + O(\eta_{k-1})) = \zeta \eta_{k-1} (1 + O(\eta_{k-1})).
\end{equation*}
Since we already know that $\eta_{k-1} \leq C \cdot B_1^{k-1}$, this implies that
\begin{equation*}
\eta_k = \zeta \eta_{k-1} (1 + O(B_1^k)).
\end{equation*}
Now it follows that the infinite product 
\begin{equation*}
\lambda_1 = \eta_0 \prod_{j \geq 1} \frac{\eta_j}{\zeta \eta_{j-1}} = \lim_{k \to \infty} \eta_0 \prod_{j = 1}^k \frac{\eta_j}{\zeta \eta_{j-1}} = \lim_{k \to \infty} \zeta^{-k} \eta_k
\end{equation*}
converges. The error bound follows from noting that
\begin{equation*}
\zeta^{-k} \eta_k = \lambda_1 \prod_{j \geq k+1} \frac{\zeta \eta_{j-1}}{\eta_j} = \lambda_1 \prod_{j \geq k+1} (1 + O(B_1^j)). \qedhere
\end{equation*}
\end{proof}

Next, we consider the expression in~\eqref{eq:Jacdet} and determine the asymptotic behaviour of its parts.
\begin{lemma}\label{lem:sum-product-refined}
  For large enough $h$ and a fixed constant $B_3 < 1$, we have
  \begin{equation*}
    1 + \sum_{k = 2}^{h}\prod_{j = k}^{h} (\rho_h \Phi'(\eta_{h, j}))
    = \frac{1}{1 - \zeta} + O(B_3^{h})
  \end{equation*}
and
    \begin{equation*}
      \prod_{j=1}^{h}\rho_h\Phi'(\eta_{h, j}) = \lambda_2 \zeta^h (1 + O(B_3^h)),
    \end{equation*}
    where $\lambda_2 := \prod_{j\geq 1}\frac{\Phi'(\eta_j)}{\Phi'(0)}$. 
\end{lemma}
\begin{proof}
Note that
\begin{equation*}
\prod_{j = k}^{h} (\rho_h \Phi'(\eta_{h, j})) = \rho_h^{h-k+1} \prod_{j = k}^{h} \Phi'(\eta_{h,j}).
\end{equation*}
In view of Lemma~\ref{lem:exp-convergence}, we have $\rho_h^{h-k+1} = \rho^{h-k+1} (1+O(B_2^h))^{h-k+1} = \rho^{h-k+1}(1+ O(hB_2^h))$, uniformly in $k$.
Moreover, Lemma~\ref{lem:exp-bound} yields $\Phi'(\eta_{h,j}) = \Phi'(0) + O(\eta_{h,j}) = \Phi'(0) + O(B_1^j)$, uniformly in $h$. Thus
\begin{equation*}
\prod_{j = k}^{h} \Phi'(\eta_{h,j}) = \Phi'(0)^{h-k+1} \prod_{j=k}^h (1 + O(B_1^j)) = \Phi'(0)^{h-k+1} (1 + O(B_1^k)).
\end{equation*}
Hence the expression simplifies to
  \begin{equation*}
    1 + \sum_{k = 2}^{h}\prod_{j = k}^{h} (\rho_h \Phi'(\eta_{h, j})) = 1 + (1 + O(hB_2^h))  \sum_{k = 2}^h \zeta^{h-k+1}(1 + O(B_1^k)).
  \end{equation*}
  Since $\zeta < 1$ and $B_1 < 1$, we can simply evaluate the geometric series,
and the expression further simplifies to 
  \begin{equation*}
    1 + \sum_{k = 2}^h \zeta^{h-k+1} + O(B_3^h) =
\frac{1-\zeta^h}{1-\zeta} + O(B_3^h) = \frac{1}{1-\zeta} + O(B_3^h)
  \end{equation*}
  for an appropriately chosen $B_3 < 1$.  This proves the first statement.  For
  the second statement, we also use Lemma~\ref{lem:exp-convergence}, along with
  the monotonicity of $\Phi'$ and the assumption that $\Phi'(0) \neq 0$, which
  implies that $\Phi'(\eta_j)$ is bounded away from $0$. This yields
    \begin{align*}
      \prod_{j=1}^{h}\rho_h\Phi'(\eta_{h, j}) = \prod_{j=1}^{h}(\rho+O(B_2^h))(\Phi'(\eta_j)+O(B_2^h)) = 
      \rho^h (1 + O(hB_2^h))\prod_{j=1}^h\Phi'(\eta_j).
    \end{align*}
    Since $\Phi'(\eta_j) = \Phi'(0) + O(\zeta^j)$ (by Lemma~\ref{lem:eta-asymp}), the product that defines $\lambda_2$ converges. So we can rewrite the product term as 
    \begin{equation*}
        \prod_{j=1}^h\Phi'(\eta_j) = \Phi'(0)^h \prod_{j=1}^h\frac{\Phi'(\eta_j)}{\Phi'(0)} = \lambda_2 \Phi'(0)^h \prod_{j\geq h+1}\frac{\Phi'(0)}{\Phi'(\eta_j)},
    \end{equation*}
    and thus, using again the estimate $\Phi'(\eta_j) = \Phi'(0) + O(\zeta^j)$ on the remaining product,
    \begin{align*}
      \prod_{j=1}^{h}\rho_h\Phi'(\eta_{h, j})
      & = \lambda_2 \zeta^h (1 + O(hB_2^h))(1+O(\zeta^h)).
    \end{align*}
This proves the desired formula for a suitable choice of $B_3$.
\end{proof}

\begin{corollary}\label{cor:Phi-prime}
    For sufficiently large $h$, we have that 
    \begin{equation}\label{eq:rho-Phi-prime-exp-fast-plus-error}
        \rho_h\Phi'(\eta_{h, 0}) = 1 + \lambda_2 (1 - \zeta)\zeta^h (1+ O(B_3^{h})),
    \end{equation}
    where $\lambda_2$ and $B_3$ are as in Lemma~\ref{lem:sum-product-refined}.
\end{corollary}
\begin{proof}
    Taking the asymptotic formulas from the statement of Lemma~\ref{lem:sum-product-refined} and applying them to~\eqref{eq:Jacdet} we obtain the formula after solving for $\rho_h\Phi'(\eta_{h, 0})$.
\end{proof}

In the proof of Lemma~\ref{lem:exp-convergence} we used the bound $\eta_{h, h} = O(B_1^h)$ (obtained from Lemma~\ref{lem:exp-bound}). In order to refine the process, we need a more precise estimate. 

\begin{lemma}\label{lem:etahh}
    For sufficiently large $h$ and a fixed constant $B_4 < 1$, we have that 
    \begin{equation}\label{eq:etahh-formula}
      \eta_{h, h} = \lambda_1(1 - \zeta) \zeta^{h} (1 + O(B_4^h)),
    \end{equation}
    where $\lambda_1$ is as defined in Lemma~\ref{lem:eta-asymp}.
\end{lemma}
\begin{proof}
  Pick some $\alpha \in (0,1)$ in such a way that $\zeta^{\alpha} > B_2$, with
  $B_2$ as in Lemma~\ref{lem:exp-convergence}, and set
  $m = \lfloor \alpha h \rfloor$.  From Lemma~\ref{lem:eta-asymp}, we know that
  $\eta_m = \Theta( \zeta^{\alpha h})$. By Lemma~\ref{lem:exp-convergence},
  $\eta_{h,m} = \eta_m + O(B_2^h)$, so by our choice of $\alpha$ there is some
  $B_4 < 1$ such that $\eta_{h,m} = \eta_m (1 + O(B_4^h))$ for sufficiently
  large $h$.

Next, recall from~\eqref{eq:syst2} that
\begin{equation*}
\eta_{h,k} = \rho_h \big( \Phi(\eta_{h,k-1}) - \Phi(\eta_{h,h}) \big).
\end{equation*}
By the mean value theorem, there is some $\xi_{h,k} \in (\eta_{h,h},\eta_{h,k-1})$ such that
\begin{equation*}
\eta_{h,k} = \rho_h (\eta_{h,k-1} - \eta_{h,h}) \Phi'(\xi_{h,k}) = \rho_h (\eta_{h,k-1} - \eta_{h,h}) (\Phi'(0) + O(\eta_{h,k-1})).
\end{equation*}
Assume now that $k \geq m$, so that $\eta_{h,k-1} = O(B_1^{\alpha h})$ by Lemma~\ref{lem:exp-bound}. Moreover, $\rho_h = \rho + O(B_2^h)$ by Lemma~\ref{lem:exp-convergence}. So with $B = \max(B_2,B_1^{\alpha})$, it follows that
\begin{equation*}
\eta_{h,k} = \zeta (\eta_{h,k-1} - \eta_{h,h}) (1 + O(B^h)),
\end{equation*}
uniformly for all $k \geq m$. Rewrite this as
\begin{equation*}
\eta_{h,k-1} = \eta_{h,h} + \frac{\eta_{h,k}}{\zeta} (1 + O(B^h)).
\end{equation*}
Iterate this $h-m$ times to obtain
\begin{align*}
\eta_{h,m} &= \sum_{j=0}^{h-m} \frac{\eta_{h,h}}{\zeta^j} (1 + O(B^h))^j \\
&= \eta_{h,h} \zeta^{-(h-m)} \frac{1 - \zeta^{h-m+1}}{1-\zeta} (1 + O(h B^h)).
\end{align*}
Now recall that $\eta_{h,m} = \eta_m (1 + O(B_4^h))$, and that
$\eta_m = \lambda_1 \zeta^m (1 + O(B_1^{\alpha h}))$ by
Lemma~\ref{lem:eta-asymp}. Plugging all this in and solving for $\eta_{h,h}$,
we obtain~\eqref{eq:etahh-formula}, provided that $B_4$ was also chosen to be
greater than $B$ and $\zeta^{1-\alpha}$.
\end{proof}
  
Now we can make use of this asymptotic formula for $\eta_{h,h}$ in order to obtain a refined estimate for $\eta_{h, 0}$. 

\begin{proposition}\label{prop:rho-convergence}
  For a fixed constant $B_5 < 1$ and large enough $h$, we have that
  \begin{equation}\label{eq:eta-precise}
    \eta_{h, 0} = \tau + \frac{(1-\zeta)(\Phi(\tau) \lambda_2 - \Phi'(0) \lambda_1)}{\tau\Phi''(\tau)}\zeta^h + O((\zeta B_5)^h)
  \end{equation}
  and
  \begin{equation}\label{eq:rho-precise}
    \rho_h = \rho + \frac{\lambda_1(1-\zeta)}{\Phi(\tau)}\zeta^{h+1} + O((\zeta B_5)^h),
  \end{equation}
  where $\lambda_1$ and $\lambda_2$ are as in Lemma~\ref{lem:eta-asymp} and Lemma~\ref{lem:sum-product-refined} respectively.
\end{proposition}
\begin{proof}
From~\eqref{eq:syst1} and~\eqref{eq:syst2} with $k = 1$, we have
\begin{equation}\label{eq:etah0-etahh}
\eta_{h,0} = \rho_h \big( \Phi(\eta_{h,0}) - \Phi(\eta_{h, h}) + 1 \big).
\end{equation}
By means of Taylor expansion and Lemma~\ref{lem:etahh}, we get
\begin{equation*}
    \eta_{h, 0} = \rho_h \big( \Phi(\eta_{h, 0}) - \Phi'(0)\eta_{h, h} + O(\eta_{h, h}^2) \big).
\end{equation*}
We multiply this by~\eqref{eq:rho-Phi-prime-exp-fast-plus-error} and divide through by $\rho_h$ to obtain
\begin{equation}\label{eq:eta-h0-final-implicit}
    \eta_{h, 0} \Phi'(\eta_{h, 0}) = \big( \Phi(\eta_{h, 0}) - \Phi'(0)\eta_{h, h} + O(\eta_{h, h}^2) \big) \big(1 + \lambda_2 (1 - \zeta) \zeta^{h} (1+ O(B_3^{h})) \big)
\end{equation}
or, with $H(x) = x\Phi'(x) - \Phi(x)$,
\begin{equation*}
    H(\eta_{h,0}) = \big({-} \Phi'(0)\eta_{h, h} + O(\eta_{h, h}^2) \big) (1 + O(\zeta^{h})) 
+ \Phi(\eta_{h,0}) \lambda_2 (1 - \zeta) \zeta^h (1+ O(B_3^{h})).
\end{equation*}
We plug in the asymptotic formula for $\eta_{h,h}$ from Lemma~\ref{lem:etahh}
and also note that
$\Phi(\eta_{h,0}) = \Phi(\tau + O(B_2^h)) = \Phi(\tau) + O(B_2^h)$ by
Lemma~\ref{lem:exp-convergence}. This gives us
\begin{equation}\label{eq:etah0_final_implicit}
    H(\eta_{h,0}) = (\Phi(\tau) \lambda_2 - \Phi'(0) \lambda_1 )(1 - \zeta) \zeta^h + O((\zeta B_5)^h),
\end{equation}
where $B_5 = \max(\zeta,B_2,B_3,B_4)$.  Now note that the function $H$ is
increasing (on the positive real numbers within the radius of convergence of $\Phi$) with derivative $H'(x) = x \Phi''(x)$ and a unique zero at
$\tau$. So by inverting~\eqref{eq:etah0_final_implicit}, we finally end up with
\begin{equation*}
\eta_{h,0} = \tau + \frac{1}{H'(\tau)} (\Phi(\tau) \lambda_2 - \Phi'(0) \lambda_1 )(1 - \zeta) \zeta^h + O((\zeta B_5)^h),
\end{equation*}
completing the proof of the first formula.
Now we return to~\eqref{eq:rho-Phi-prime-exp-fast-plus-error}, which gives us
\begin{equation*}
\rho_h = \frac{1 + \lambda_2 (1 - \zeta)\zeta^h (1+ O(B_3^{h}))}{\Phi'(\eta_{h, 0})} \\
= \frac{1 + \lambda_2 (1 - \zeta)\zeta^h (1+ O(B_3^{h}))}{\Phi'(\tau) + \Phi''(\tau)(\eta_{h,0}-\tau) + O((\eta_{h,0}-\tau)^2)}.
\end{equation*}
Plugging in~\eqref{eq:eta-precise} and simplifying by means of the identities $\rho\Phi(\tau) = \tau$ and $\rho \Phi'(\tau) = 1$ now yields~\eqref{eq:rho-precise}.
\end{proof}

\subsection{Proof of Theorem~\ref{thm:main_outdeg1_allowed}}\label{sec:prowags-exp}

We are now finally ready to apply Theorem~\ref{thm:pro-wags-1} and
Theorem~\ref{thm:pro-wags-2}. The generating functions
$Y_h(z) := Y_{h,0}(z) = Y_{h,1}(z) + z$ were defined precisely in such a way
that $y_{h,n} = [z^n] Y_h(z)$ is the number of $n$-vertex trees for which the
maximum protection number is less than or equal to $h$. Thus the random
variable $X_n$ in Theorem~\ref{thm:pro-wags-1} becomes the maximum protection
number of a random $n$-vertex tree. Condition~(\ref{bullet:thm-1-2}) of
Theorem~\ref{thm:pro-wags-1} is satisfied in view of
Proposition~\ref{prop:analytic_properties_of_Yh}. Condition~(\ref{bullet:thm-1-1})
holds by Proposition~\ref{prop:rho-convergence} with $\zeta = \rho \Phi'(0)$
and
\begin{equation}\label{eq:expontential-kappa}
  \kappa = \frac{\lambda_1 (1-\zeta) \zeta}{\rho \Phi(\tau)} =
\frac{\lambda_1(1-\zeta)\zeta}{\tau},
\end{equation}  
where $\lambda_1$ is as defined in Lemma~\ref{lem:eta-asymp} and 
we recall the definition of $\zeta$
as $\rho\Phi'(0)$. This already proves the first part of
Theorem~\ref{thm:main_outdeg1_allowed}.

We can also apply Theorem~\ref{thm:pro-wags-2}: Note that the maximum
protection number of a tree with size $n$ is no greater than $n-1$, thus
$y_{h, n} = y_n$ for $h \geq n-1$, and an appropriate choice of constant for
Condition (\ref{bullet:thm-2-1}) in Theorem~\ref{thm:pro-wags-2} would be
$K = 1$. Conditions~(\ref{bullet:thm-2-2}) and~(\ref{bullet:thm-2-3}) are still
covered by Proposition~\ref{prop:analytic_properties_of_Yh}. Hence
Theorem~\ref{thm:pro-wags-2} applies, and the second part of
Theorem~\ref{thm:main_outdeg1_allowed} follows.

\section{The double-exponential case: \texorpdfstring{$w_1 = 0$}{w1 is 0}}\label{sec:case-double-exp}

\subsection{Asymptotics of the singularities}\label{sec:singularity-asymptotics-double-exp}

In Section~\ref{sec:singularity-asymptotics-exp}, it was crucial in most of our
asymptotic estimates that $w_1 = \Phi'(0) \neq 0$. In this section we assume
that $w_1 = \Phi'(0) = 0$ and define $r$ to be the smallest positive outdegree with
nonzero weight:
\begin{equation*}
r = \min \{i \in \mathbb{N}: i \geq 2 \text{ and } w_i \neq 0\} = \min\{i \in \mathbb{N}: i \geq 2 \text{ and } \Phi^{(i)}(0) \neq 0\}. 
\end{equation*}
Our goal will be to determine the asymptotic behaviour of $\rho_h$ in this
case, based again on the system of equations that is given
by~\eqref{eq:syst1},~\eqref{eq:syst2} and~\eqref{eq:Jacdet}. Once again, $B_i$'s
will always denote positive constants with $B_i < 1$ (different from those in
the previous section, but for simplicity we restart the count at $B_1$) that
depend on the specific family of simply generated trees, but nothing else.

No part of the proof of Lemma~\ref{lem:exp-bound} depends on $\Phi'(0) \neq 0$
and thus it also holds in the case which we are currently working in, so we
already have an exponential bound on $\eta_{h,k}$. However, this bound is loose
if $\Phi'(0) = 0$, and so we determine a tighter bound.

\begin{lemma}\label{lem:doubly-exp-bound}
    There exist positive constants $C$ and $B_1$ with $B_1 < 1$ such that $\eta_{h,k} \leq C B_1^{r^k}$ for all sufficiently large $h$ and all $k$ with $0 \leq k \leq h$.
\end{lemma}
\begin{proof}
    From \eqref{eq:syst2}, we have that $\eta_{h, k} = \rho_h\Phi(\eta_{h, k-1}) - \rho_h\Phi(\eta_{h, h})$.
    Using the Taylor expansion about 0, this gives, for some $\xi_{h, k-1} \in (0, \eta_{h, k-1})$,
\begin{align*}
\eta_{h, k} &= \rho_h\Big(\Phi(0) + \frac{\Phi^{(r)}(\xi_{h, k-1})}{r!}\eta_{h, k-1}^r\Big) - \rho_h\Phi(\eta_{h, h}) \\
&\leq \rho_h \frac{\Phi^{(r)}(\xi_{h, k-1})}{r!}\eta_{h, k-1}^r \leq \rho_h \frac{\Phi^{(r)}(\eta_{h, 1})}{r!}\eta_{h, k-1}^r.
\end{align*}
There is a constant $M$ such that
$\rho_h \frac{\Phi^{(r)}(\eta_{h, 1})}{r!} \leq M$ for all sufficiently large
$h$, since we already know that $\rho_h$ and $\eta_{h,1}$ converge. So for
sufficiently large $h$, we have $\eta_{h, k} \leq M \eta_{h,k-1}^r$ for all
$k > 1$. Iterating this inequality yields
\begin{equation*}
\eta_{h,k} \leq M^{\frac{r^{k-\ell}-1}{r-1}} \eta_{h,\ell}^{r^{k-\ell}}
\end{equation*}
for $0\le \ell\le k$.
In view of the exponential bound on $\eta_{h,\ell}$ provided by Lemma~\ref{lem:exp-bound}, we can
choose $\ell$ so large that $M^{1/(r-1)} \eta_{h,\ell} \leq \frac12$ for all
sufficiently large $h$. This proves the desired bound for $k \geq \ell$ with
$B_1 = 2^{-r^{-\ell}}$ and a suitable choice of $C$ (for $k < \ell$, it is
implied by the exponential bound).
\end{proof}

Our next step is an analogue of Lemma~\ref{lem:sum-product-refined}.

\begin{lemma}\label{lem:doubly-exp-sum-product}
  For large enough $h$ and the same constant $B_1 < 1$ as in the previous lemma, we have
  \begin{equation*}
    1 + \sum_{k = 2}^{h}\prod_{j = k}^{h} (\rho_h \Phi'(\eta_{h, j})) = 1 + O(B_1^{r^h})
  \end{equation*}
and
    \begin{equation*}
      \prod_{j=1}^{h}\rho_h\Phi'(\eta_{h, j}) = O(B_1^{r^h}).
    \end{equation*}
\end{lemma}

\begin{proof}
  We already know that $\rho_h \Phi'(\eta_{h,1})$ converges to
  $\rho \Phi'(\tau - \rho) < 1$, so for sufficiently large $h$ and some
  $q < 1$, we have
  $\rho_h\Phi'(\eta_{h, j}) \leq \rho_h \Phi'(\eta_{h,1}) \leq q$ for all
  $j \geq 1$. It follows that
\begin{equation*}
\sum_{k = 2}^{h}\prod_{j = k}^{h} (\rho_h \Phi'(\eta_{h, j})) \leq \sum_{k=2}^h q^{h-k} \rho_h \Phi'(\eta_{h,h}) \leq \frac{1}{1-q} \rho_h \Phi'(\eta_{h,h})
\end{equation*}
and
    \begin{equation*}
      \prod_{j=1}^{h}\rho_h\Phi'(\eta_{h, j}) \leq q^{h-1} \rho_h \Phi'(\eta_{h,h}).
    \end{equation*}
Now both statements  follow from the fact that $\Phi'(\eta_{h,h}) = \Phi'(0) + O(\eta_{h,h}) = O(\eta_{h,h})$ and the previous lemma.
\end{proof}

Taking the results from Lemma~\ref{lem:doubly-exp-sum-product} and applying them to \eqref{eq:Jacdet}, we find that 
\begin{equation}\label{eq:doubly-exp-first-rho}
  \rho_h\Phi'(\eta_{h, 0}) = 1 + O\big(B_1^{r^{h}}\big). 
\end{equation}
Additionally note that using Lemma~\ref{lem:doubly-exp-bound} and Taylor expansion, we have that
\begin{equation*}
  \Phi(\eta_{h, h}) = 1 + O(B_1^{r^{h}}).
\end{equation*}
Now recall that~\eqref{eq:syst1} and~\eqref{eq:syst2} yield (see~\eqref{eq:etah0-etahh})
\begin{equation}\label{eq:eta-h0-equation}
\eta_{h,0} = \rho_h \big( \Phi(\eta_{h,0}) - \Phi(\eta_{h, h}) + 1 \big),
\end{equation}
which now becomes
\begin{equation}\label{eq:doubly-exp-first-eta}
  \eta_{h, 0} = \rho_h\Phi(\eta_{h, 0}) + O\big(B_1^{r^{h}}\big).
\end{equation}

Taking advantage of the expressions in~\eqref{eq:doubly-exp-first-rho} and~\eqref{eq:doubly-exp-first-eta}, we can now prove doubly exponential convergence of 
$\rho_h$ and $\eta_{h, 0}$ (using the approach of Proposition~\ref{prop:rho-convergence}).

\begin{lemma}\label{lem:doubly-exp-convergence}
  For large enough $h$, it holds that 
  \begin{equation*}
    \rho_h = \rho + O\big(B_1^{r^{h}}\big) \qquad \text{and} \qquad \eta_{h, 0} = \tau + O\big(B_1^{r^{h}}\big).
  \end{equation*} 
and thus also $\eta_{h,1} = \eta_{h,0} - \rho_h = \eta_1 + O(B_1^{r^h})$.
\end{lemma}
\begin{proof}
  Multiplying~\eqref{eq:doubly-exp-first-rho} and~\eqref{eq:doubly-exp-first-eta} and dividing by $\rho_h$ yields
  \begin{equation*}
    \eta_{h, 0}\Phi'(\eta_{h, 0}) = \Phi(\eta_{h, 0}) + O\big(B_1^{r^{h}}\big).
  \end{equation*}
  As in the proof of Proposition~\ref{prop:rho-convergence}, we observe that the
  function $H(x) = x \Phi'(x) - \Phi(x)$ is increasing (on the positive real numbers within the radius of convergence of $\Phi$) with derivative
  $H'(x) = x \Phi''(x)$ and a unique zero at $\tau$. So it follows from this
  equation that $\eta_{h, 0} = \tau + O\big(B_1^{r^{h}}\big)$.  Using this
  estimate for $\eta_{h, 0}$ in~\eqref{eq:doubly-exp-first-rho} it follows that
  $\rho_h = \rho + O\big(B_1^{r^{h}}\big)$.
\end{proof}

As in the previous section, we will approximate $\eta_{h,k}$ by $\eta_k$,
defined recursively by $\eta_0 = \tau$ and
$\eta_k = \rho (\Phi(\eta_{k-1}) - 1)$. As it turns out, this approximation is
even more precise in the current case.

\begin{lemma}\label{lem:etah-k-to-eta-k}
For a fixed constant $B_2 < 1$ and sufficiently large $h$, we have that 
  \begin{equation*}
    \eta_{h, k} = \eta_{k}(1 + O(B_2^{r^h})),
  \end{equation*}
uniformly for all $0 \leq k \leq h$.
\end{lemma}

\begin{proof}
Recall that, by~\eqref{eq:syst2}, $\eta_{h, k} = \rho_h\Phi(\eta_{h, k-1}) - \rho_h\Phi(\eta_{h, h})$. By Taylor expansion, we find that
\begin{equation*}
\eta_{h, k} = \rho_h\Phi(\eta_{h, k-1}) - \rho_h + O(\eta_{h,h}^r).
\end{equation*}
Since $\eta_{h,k} \geq \eta_{h,h}$, we have
$\eta_{h,k} - O(\eta_{h,h}^r) = \eta_{h,k} (1 - O(\eta_{h,h}^{r-1}))$. Now we
use the estimates $\eta_{h,h} = O(B_1^{r^h})$ from
Lemma~\ref{lem:doubly-exp-bound} and $\rho_h = \rho + O(B_1^{r^h})$ from
Lemma~\ref{lem:doubly-exp-convergence} to obtain
\begin{equation*}
\eta_{h, k} = \rho(\Phi(\eta_{h, k-1}) - 1) \big(1+O(B_1^{r^h})\big).
\end{equation*}
We compare this to
\begin{equation*}
\eta_k = \rho(\Phi(\eta_{k-1}) - 1).
\end{equation*}
Taking the logarithm in both these equations and subtracting yields
\begin{equation}\label{eq:etahk-versus-etak}
\log \frac{\eta_{h,k}}{\eta_k} = \log (\Phi(\eta_{h, k-1}) - 1) - \log (\Phi(\eta_{k-1}) - 1) + O(B_1^{r^h}).
\end{equation}

For large enough $h$, we can assume that $\eta_{h,1} \leq \tau$ and thus $\eta_{h,k} \leq \tau$ for all $k \geq 1$. The auxiliary function
\begin{equation*}
\Psi_1(u) = \log \big( \Phi(e^u) - 1 \big)
\end{equation*}
is continuously differentiable on $(-\infty,\log(\tau)]$. Since
$\lim_{u \to -\infty} \Psi_1'(u) = r$, as one easily verifies,
$\abs{\Psi_1'(u)}$ must be bounded by some constant $K$ for all $u$ in this
interval, thus $\abs{\Psi_1(u+v) - \Psi_1(u)} \leq K\abs{v}$ whenever
$u,u+v \leq \log(\tau)$. We apply this with $u+v = \log \eta_{h,k-1}$ and
$u = \log \eta_{k-1}$ to obtain
\begin{equation*}
\Big| \log (\Phi(\eta_{h, k-1}) - 1) - \log (\Phi(\eta_{k-1}) - 1) \Big| \leq K \Big| \log \frac{\eta_{h,k-1}}{\eta_{k-1}} \Big|.
\end{equation*}
Plugging this into~\eqref{eq:etahk-versus-etak} yields
\begin{equation}\label{eq:double-exponential-iteration}
\Big| \log \frac{\eta_{h,k}}{\eta_k} \Big| \leq K \Big| \log \frac{\eta_{h,k-1}}{\eta_{k-1}} \Big| + O(B_1^{r^h}).
\end{equation}
We already know that $\big| \log \frac{\eta_{h,0}}{\eta_0} \big| = O(B_1^{r^h})$ and
$\big| \log \frac{\eta_{h,1}}{\eta_1} \big| = O(B_1^{r^h})$ in view of Lemma~\ref{lem:doubly-exp-convergence}. Iterating~\eqref{eq:double-exponential-iteration} gives us
\begin{equation*}
\Big| \log \frac{\eta_{h,k}}{\eta_k} \Big| = O\big((1+K+K^2+\cdots+K^k) B_1^{r^h}),
\end{equation*}
which implies the statement for any $B_2 > B_1$.
\end{proof}

The next lemma parallels Lemma~\ref{lem:eta-asymp}.

\begin{lemma}\label{lem:eta-asymp-double-exp}
There exist positive constants $\lambda_1$ and $\mu < 1$ such that
\begin{equation*}
\eta_k = \lambda_1 \mu^{r^k}\big(1 + O(B_1^{r^k})\big),
\end{equation*}
with the same constant $B_1$ as in Lemma~\ref{lem:doubly-exp-bound}.
\end{lemma}

\begin{proof}
Note that Lemma~\ref{lem:doubly-exp-bound} trivially implies that $\eta_k=O(B_1^{r^k})$.
From the recursion
\begin{equation*}
\eta_k = \rho (\Phi(\eta_{k-1}) - 1),
\end{equation*}
we obtain, by the properties of $\Phi$,
\begin{equation*}
\eta_k = \frac{\rho \Phi^{(r)}(0)\eta_{k-1}^r}{r!} (1 + O(\eta_{k-1})).
\end{equation*}
Set 
\begin{equation}\label{eq:lambda1}
\lambda_1 = \Bigl( \frac{\rho \Phi^{(r)}(0)}{r!}\Bigr)^{-1/(r-1)} = (\rho w_r)^{-1/(r-1)}
\end{equation} 
and divide both sides by $\lambda_1$ to obtain
\begin{equation*}
\frac{\eta_k}{\lambda_1} = \Big(\frac{\eta_{k-1}}{\lambda_1} \Big)^r (1 + O(\eta_{k-1})).
\end{equation*}
Let us write $e^{\theta_{k-1}}$ for the final factor, where $\theta_{k-1} = O(\eta_{k-1})$. Taking the logarithm yields
\begin{equation*}
\log \frac{\eta_k}{\lambda_1} = r \log \frac{\eta_{k-1}}{\lambda_1} + \theta_{k-1}.
\end{equation*}
We iterate this recursion $k$ times to obtain
\begin{align*}
\log \frac{\eta_k}{\lambda_1} &= r^k \log \frac{\eta_0}{\lambda_1} + \sum_{j=0}^{k-1} r^{k-1-j} \theta_j \\
&= r^k \Big( \log \frac{\eta_0}{\lambda_1} + \sum_{j=0}^{\infty} r^{-1-j} \theta_j \Big) - \sum_{j=k}^{\infty} r^{k-1-j} \theta_j.
\end{align*}
The infinite series converge in view of the estimate
$\theta_j = O(\eta_j) = O(B_1^{r^j})$ that we get from
Lemma~\ref{lem:doubly-exp-bound}. Moreover, we have
$\sum_{j=k}^{\infty} r^{k-1-j} \theta_j = O(B_1^{r^k})$ by the same bound. The
result follows upon taking the exponential on both sides and multiplying by
$\lambda_1$, setting
\begin{equation}\label{eq:mu}
\mu := \exp \Big( \log \frac{\eta_0}{\lambda_1} + \sum_{j=0}^{\infty} r^{-1-j} \theta_j \Big) = \frac{\eta_0}{\lambda_1} \prod_{j=0}^{\infty} e^{\theta_j/r^{j+1}}.
\end{equation}
Note that $\mu<1$ because we already know that $\eta_k=O(B_1^{r^k})$.
\end{proof}

In order to further analyse
the behaviour of the product $\prod_{j = 1}^h (\rho_h \Phi'(\eta_{h, j}))$
in~\eqref{eq:Jacdet}, we need one more short lemma.

\begin{lemma}\label{lem:Phi'-convergence}
For sufficiently large $h$, we have that 
  \begin{equation*}
    \Phi'(\eta_{h, k}) = \Phi'(\eta_k)\big(1 + O(B_2^{r^h})\big),
  \end{equation*}
uniformly for all $1 \leq k \leq h$, with the same constant $B_2$ as in Lemma~\ref{lem:etah-k-to-eta-k}.
\end{lemma}
\begin{proof}
  Again, we can assume that $h$ is so large that $\eta_{h,k} \leq \tau$ for all
  $k \geq 1$. The auxiliary function $\Psi_2(u) = \log (\Phi'(e^u))$ is
  continuously differentiable on $(-\infty,\log \tau]$ and satisfies
  $\lim_{u \to -\infty} \Psi_2'(u) = r-1$. Thus its derivative is also bounded,
  and the same argument as in Lemma~\ref{lem:etah-k-to-eta-k} shows that
\begin{equation*}
\Big|\log \frac{\Phi'(\eta_{h,k})}{\Phi'(\eta_k)} \Big| \leq K \Big| \log \frac{\eta_{h,k}}{\eta_k} \Big|
\end{equation*}
for some positive constant $K$. Now the statement follows from Lemma~\ref{lem:etah-k-to-eta-k}.
\end{proof}

\begin{lemma}\label{lem:refined-product}
  There exist positive constants $\lambda_2,\lambda_3$ and $B_3 < 1$ such that, for large enough $h$, 
  \begin{equation}\label{eq:doubly-exp-product-2}
    \prod_{j = 1}^h (\rho_h \Phi'(\eta_{h, j})) = \lambda_2 \lambda_3^h\mu^{r^{h+1}} \big(1 + O(B_3^{r^h})\big),
  \end{equation}
with $\mu$ as in Lemma~\ref{lem:eta-asymp-double-exp}.
\end{lemma}
\begin{proof}
First, observe that
  \begin{equation*}
    \prod_{j = 1}^h (\rho_h \Phi'(\eta_{h, j})) = \Big(\rho \big(1 + O(B_1^{r^h})\big) \Big)^h \prod_{j=1}^h \Big( \Phi'(\eta_j) \big(1 + O(B_2^{r^h})\Big)
 = \rho^h \Big( \prod_{j=1}^h \Phi'(\eta_j) \Big) \big(1 + O(h B_2^{r^h})\big)
  \end{equation*}
in view of Lemma~\ref{lem:doubly-exp-convergence} and Lemma~\ref{lem:Phi'-convergence} (recall that $B_2 > B_1$).
Next, Taylor expansion combined with Lemma~\ref{lem:eta-asymp-double-exp} gives us
\begin{equation*}
\Phi'(\eta_k) = \frac{\Phi^{(r)}(0)}{(r-1)!} \eta_k^{r-1} (1 + O(\eta_k)) = \frac{\Phi^{(r)}(0)\lambda_1^{r-1}}{(r-1)!} \mu^{(r-1)r^k} \big(1 + O(B_1^{r^k})\big).
\end{equation*}
Set $\lambda_3 := \rho \frac{\Phi^{(r)}(0)\lambda_1^{r-1}}{(r-1)!} = rw_r \rho \lambda_1^{r-1}$, so that
\begin{equation*}
\Phi'(\eta_k) = \frac{\lambda_3}{\rho} \mu^{(r-1)r^k} \big(1 + O(B_1^{r^k})\big).
\end{equation*}
It follows that the infinite product
\begin{equation*}
\Pi:= \prod_{j = 1}^{\infty} \frac{\rho \Phi'(\eta_j)}{\lambda_3 \mu^{(r-1)r^j}}
\end{equation*}
converges, and that
\begin{equation*}
\prod_{j = 1}^{h} \frac{\rho \Phi'(\eta_j)}{\lambda_3 \mu^{(r-1)r^j}} = \Pi \big(1 + O(B_1^{r^h})\big).
\end{equation*}
Consequently,
\begin{equation*}
\prod_{j = 1}^{h} \Phi'(\eta_j) = \Pi \Big( \frac{\lambda_3}{\rho} \Big)^h \mu^{r^{h+1}-r} \big(1 + O(B_1^{r^h})\big).
\end{equation*}
Putting everything together, the statement of the lemma follows with $\lambda_2 = \Pi \mu^{-r}$ and a suitable choice of $B_3 > B_2$.
\end{proof}

With this estimate for the product term in the determinant of the Jacobian~\eqref{eq:Jacdet}, and the estimate for the sum term 
from Lemma~\ref{lem:doubly-exp-sum-product}, we can now obtain a better asymptotic formula for $\rho_h\Phi'(\eta_{h, 0})$ than that which 
was obtained in~\eqref{eq:doubly-exp-first-rho}.
For large enough $h$, we have that 
  \begin{equation}\label{eq:doubly-exp-second-rho}
    \rho_h\Phi'(\eta_{h, 0}) = 1 + \lambda_2 \lambda_3^h \mu^{r^{h+1}} \big(1 + O(B_3^{r^h})\big).
  \end{equation}
  For the error term, recall that $B_1 < B_3$. Moreover,
  combining Lemmata~\ref{lem:etah-k-to-eta-k} and~\ref{lem:eta-asymp-double-exp} leads to 
  \begin{equation*}
    \eta_{h, h} = \lambda_1 \mu^{r^h}\big(1 + O(B_2^{r^h})\big)
  \end{equation*}
  since $B_2$ was chosen to be greater than $B_1$,
which we can apply to~\eqref{eq:eta-h0-equation}:
\begin{align}
\eta_{h,0} &= \rho_h \big( \Phi(\eta_{h,0}) - \Phi(\eta_{h, h}) + 1 \big) \nonumber \\
&= \rho_h \Big( \Phi(\eta_{h,0}) - \frac{\Phi^{(r)}(0)}{r!} \eta_{h,h}^r + O(\eta_{h,h}^{r+1}) \Big) \nonumber \\
&= \rho_h \Big( \Phi(\eta_{h,0}) - w_r \lambda_1^r \mu^{r^{h+1}}\big(1 + O(B_3^{r^h})\big) \Big) \label{eq:doubly-exp-second-eta}
\end{align}
since $B_3$ was chosen to be greater than $B_1$ (and thus also $\mu$) and
$B_2$. As we did earlier to obtain Lemma~\ref{lem:doubly-exp-convergence}, we
multiply the two equations~\eqref{eq:doubly-exp-second-rho}
and~\eqref{eq:doubly-exp-second-eta} and divide by $\rho_h$ to find that
\begin{equation*}
\eta_{h,0} \Phi'(\eta_{h, 0}) = \Big( \Phi(\eta_{h,0}) - w_r \lambda_1^r \mu^{r^{h+1}}\big(1 + O(B_3^{r^h})\big) \Big) \Big( 1 + \lambda_2 \lambda_3^h \mu^{r^{h+1}} \big(1 + O(B_3^{r^h})\big) \Big).
\end{equation*}
From this, the following result follows now in exactly the same way as Proposition~\ref{prop:rho-convergence} follows from~\eqref{eq:eta-h0-final-implicit}.

\begin{proposition}\label{prop:doubly-exp-rho-convergence}
  For large enough $h$ and a fixed constant $B_4 < 1$, we have that 
  \begin{equation*}
    \eta_{h, 0} = \tau + \frac{\Phi(\tau) \lambda_2 \lambda_3^h - w_r \lambda_1^r}{\tau \Phi''(\tau)} \mu^{r^{h+1}}+ O(\mu^{r^{h+1}}B_4^{r^h})
  \end{equation*}
  and 
  \begin{equation*}
    \rho_h = \rho\Big(1 + \frac{w_r\lambda_1^r}{\Phi(\tau)}\mu^{r^{h+1}} + O(\mu^{r^{h+1}}B_4^{r^h})\Big).
  \end{equation*}
\end{proposition}

\subsection{An adapted general scheme and the proof of Theorem~\ref{thm:main_outdeg1_forbidden}}\label{sec:prowags-doubly-exp-restated}

In this final section, we will first prove Theorems~\ref{thm:pro-wags-doubly-1}
and~\ref{thm:doubly-exponential-two-points}. Then, we will be able to put all
pieces
together and prove Theorem~\ref{thm:main_outdeg1_forbidden}.

\begin{proof}[Proof of Theorem~\ref{thm:pro-wags-doubly-1}]
  We apply singularity analysis, and use the uniformity condition to obtain 
  \begin{equation*}
    y_{h,n} = \frac{A_h}{\Gamma(-\alpha)}n^{-\alpha-1}\rho_h^{-n}(1 + o(1))
  \end{equation*}
  uniformly in $h$ as $n \to \infty$ as well as 
  \begin{equation*}
    y_n = \frac{A}{\Gamma(-\alpha)}n^{-\alpha-1}\rho^{-n}(1+o(1)).
  \end{equation*}
  Since in addition $A_h \to A$ and $\rho_h = \rho(1+ \kappa \zeta^{r^{h}} + o(\zeta^{r^{h}}))$, it holds that 
  \begin{align*}
    \frac{y_{h, n}}{y_n} &= \Big(\frac{\rho_h}{\rho}\Big)^{-n}(1+o(1)) = \exp{\big({-}
\kappa n \zeta^{r^h} +o(n\zeta^{r^h})\big)}(1+o(1)) \nonumber \\
    &= \exp{\big({-}\kappa n\zeta^{r^h}(1+o(1)) + o(1)\big)}.
\end{align*}
\end{proof}

\begin{proof}[Proof of Theorem~\ref{thm:doubly-exponential-two-points}]
Fix $\epsilon > 0$. If $h \geq m_n + \epsilon = \log_r{\log_d(n)} + \epsilon$, then Theorem~\ref{thm:pro-wags-doubly-1} gives us
\begin{equation*}
  \mathbb{P}(X_n \leq h) \geq \exp{\big({-}\kappa n^{1-r^{\epsilon}} (1+o(1)) + o(1)\big)} = 1 - o(1),
\end{equation*}
thus $X_n \leq h$ with high probability. If $\{m_n\} \leq 1-\epsilon$, then this is the case for $h = \lceil m_n \rceil$, otherwise for $h = \lceil m_n \rceil + 1$.
Similarly, if $h \leq m_n - \epsilon = \log_r{\log_d(n)} - \epsilon$, then Theorem~\ref{thm:pro-wags-doubly-1} gives us
\begin{equation*}
  \mathbb{P}(X_n \leq h) \leq \exp{\big({-}\kappa n^{1-r^{-\epsilon}} (1+o(1)) + o(1)\big)} = o(1),
\end{equation*}
thus $X_n > h$ with high probability. If $\{m_n\} \geq \epsilon$, then this is the case for $h = \lfloor m_n \rfloor$, otherwise for $h = \lfloor m_n \rfloor - 1$. The statement now follows by combining the two parts.
\end{proof}

\begin{remark}\label{remark:periodic-case-2}
  As in Remark~\ref{remark:periodic-case-1}, we indicate the changes which are
  necessary for the case that the period $D$ of $\Phi$ is greater than $1$.

  Theorem~\ref{thm:pro-wags-doubly-1} only depends on singularity analysis. It
  is well known (see \cite[Remark~VI.17]{Flajolet-Sedgewick:ta:analy}) that
  singularity analysis simply introduces a factor $D$ in this situation, and as
  this factor $D$ cancels because it occurs both in the asymptotic expansions
  of $Y_{h}$ as well as $Y$, this theorem remains valid for $n\equiv 1\pmod D$.
\end{remark}

Theorem~\ref{thm:main_outdeg1_forbidden} is now an immediate consequence of
Theorem~\ref{thm:pro-wags-doubly-1} and
Theorem~\ref{thm:doubly-exponential-two-points}. In analogy to the proof of
Theorem~\ref{thm:main_outdeg1_allowed}, the analytic conditions on the
generating functions are provided by
Proposition~\ref{prop:analytic_properties_of_Yh}. The condition on the
asymptotic behaviour of $\rho_h$ is given by
Proposition~\ref{prop:doubly-exp-rho-convergence} (with $\zeta = \mu^r$). Thus
the proof of Theorem~\ref{thm:main_outdeg1_forbidden} is complete.

\bibliographystyle{plain}
\bibliography{bib/cheub}

\end{document}